\documentclass[a4]{article}
\usepackage{newtxtext}
\usepackage{amsmath,amsthm,amssymb,amscd,stmaryrd}
\usepackage{tikz-cd}
\usepackage{tikz}
\usetikzlibrary{matrix}
\usepackage{graphicx}
\usepackage{hyperref}
\usepackage{authblk}

\newtheorem{thm}{Theorem}
\newtheorem{lem}[thm]{Lemma}
\newtheorem{prop}[thm]{Proposition}
\newtheorem{cor}[thm]{Corollary}
\newtheorem*{sublem}{Sublemma}

\theoremstyle{definition}

\newtheorem{rem}[thm]{Remark}

\def\Hom{\mathop{\mathrm{Hom}}\nolimits}
\def\End{\mathop{\mathrm{End}}\nolimits}
\def\Aut{\mathop{\mathrm{Aut}}\nolimits}
\def\Out{\mathop{\mathrm{Out}}\nolimits}
\def\Inn{\mathop{\mathrm{Inn}}\nolimits}

\def\Ad{\mathop{\mathrm{Ad}}\nolimits}
\def\ad{\mathop{\mathrm{ad}}\nolimits}
\def\id{\mathop{\mathrm{id}}\nolimits}

\def\GL{\mathop{\mathrm{GL}}\nolimits}
\def\SL{\mathop{\mathrm{SL}}\nolimits}
\def\PSL{\mathop{\mathrm{PSL}}\nolimits}
\def\SO{\mathop{\mathrm{SO}}\nolimits}
\def\SU{\mathop{\mathrm{SU}}\nolimits}
\def\Sp{\mathop{\mathrm{Sp}}\nolimits}

\def\Gr{\mathop{\mathrm{Gr}}\nolimits}

\def\Isom{\mathop{\mathrm{Isom}}\nolimits}

\def\max{\mathop{\mathrm{max}}\limits}

\newcommand{\mf}[1]{\mathfrak{#1}}
\newcommand{\bb}[1]{\mathbb{#1}}
\newcommand{\mca}[1]{\mathcal{#1}}
\newcommand{\ol}[1]{\overline{#1}}
\newcommand{\df}{d_{\mathcal{F}}}

\title{Vanishing of cohomology and parameter rigidity of actions of solvable Lie groups, II}
\author{Hirokazu Maruhashi\thanks{maruhashihirokazu@gmail.com}}
\date{\empty}

\begin{document}
\maketitle

\begin{abstract}
Let $M\stackrel{\rho_0}{\curvearrowleft}S$ be a $C^\infty$ locally free action of a connected simply connected solvable Lie group $S$ on a closed manifold $M$. Roughly speaking, $\rho_0$ is parameter rigid if any $C^\infty$ locally free action of $S$ on $M$ having the same orbits as $\rho_0$ is $C^\infty$ conjugate to $\rho_0$. In this paper we prove two types of result on parameter rigidity. 

First let $G$ be a connected semisimple Lie group with finite center of real rank at least $2$ without compact factors nor simple factors locally isomorphic to $\SO_0(n,1)$ $(n\geq2)$ or $\SU(n,1)$ $(n\geq2)$, and let $\Gamma$ be an irreducible cocompact lattice in $G$. Let $G=KAN$ be an Iwasawa decomposition. We prove that the action $\Gamma\backslash G\curvearrowleft AN$ by right multiplication is parameter rigid. One of the three main ingredients of the proof is the rigidity theorems of Pansu and Kleiner--Leeb on the quasiisometries of Riemannian symmetric spaces of noncompact type. 

Secondly we show, if $M\stackrel{\rho_0}{\curvearrowleft}S$ is parameter rigid, then the zeroth and first cohomology of the orbit foliation of $\rho_0$ with certain coefficients must vanish. This is a partial converse to the results in the author's \cite{Ma2}, where we saw sufficient conditions for parameter rigidity in terms of vanishing of the first cohomology with various coefficients. 
\end{abstract}

\tableofcontents

\section*{}
This paper consists of two parts. The first part, Section \ref{introd} to Section \ref{an2}, deals with parameter rigidity of certain actions. Section \ref{introd} serves as the introduction for the first part. The second part is on necessary conditions for parameter rigidity in terms of vanishing of cohomology, and it is from Section \ref{necessary conditions} to Section \ref{hone}, where Section \ref{necessary conditions} serves as the introduction for the second part.

\section{Parameter rigidity of the action of $AN$ on $\Gamma\backslash G$}\label{introd}
Let $M\stackrel{\rho_0}{\curvearrowleft}S$ be a $C^\infty$ locally free (ie the isotropy subgroup of every point is discrete) action of a connected simply connected solvable Lie group $S$ on a closed $C^\infty$ manifold $M$. Let $\mca{F}$ be the set of all orbits of $\rho_0$, which is called {\em the orbit foliation of $\rho_0$} and actually is a $C^\infty$ foliation of $M$. We say $\rho_0$ is {\em parameter rigid} if every $C^\infty$ locally free action $M\stackrel{\rho}{\curvearrowleft}S$ with the same orbit foliation as that of $\rho_0$ is {\em parameter equivalent to $\rho_0$}. (We do not assume that $\rho$ is close to $\rho_0$ in some topology.) Here parameter equivalence between $\rho$ and $\rho_0$ means the following. There exist a diffeomorphism $F$ of $M$ and an automorphism $\Phi$ of $S$ such that: 
\begin{itemize}
\item $F\left(\rho_0(x,s)\right)=\rho\left(F(x),\Phi(s)\right)$ for all $x\in M$ and $s\in S$ 
\item the map $F$ preserves each leaf of $\mca{F}$, that is, $F(L)\subset L$ for all $L\in\mca{F}$ 
\item the map $F$ is $C^0$ homotopic to the identity map of $M$ through $C^\infty$ maps which preserve each leaf of $\mca{F}$. 
\end{itemize}

For example a linear flow on a torus is parameter rigid if and only if the velocity vector satisfies the Diophantus condition. 

In \cite{KS} and \cite{KS2} Katok and Spatzier proved the following. 

\begin{thm}[Katok--Spatzier]\label{KS}
Let $G$ be a connected semisimple Lie group with finite center of real rank at least $2$ without compact factors nor simple factors locally isomorphic to $\SO_0(n,1)$ $(n\geq2)$ or $\SU(n,1)$ $(n\geq2)$, and let $\Gamma$ be an irreducible cocompact lattice in $G$. Let $G=KAN$ be an Iwasawa decomposition. Then the action $\Gamma\backslash G\curvearrowleft A$ by right multiplication is parameter rigid. 
\end{thm}

This is proved using representation theory of semisimple Lie groups and has lead to a number of subsequent research. In this paper we prove the following, based on the above theorem and applying large scale geometry. 

\begin{thm}\label{at least 2}
Under the same assumptions as Theorem \ref{KS}, the action $\Gamma\backslash G\curvearrowleft AN$ by right multiplication is parameter rigid. 
\end{thm}

We give a proof of this theorem in Section \ref{an1} and Section \ref{an2} after recalling the results in Maruhashi \cite{Ma2} in Section \ref{prelimi}. The proof is a combination of the following three steps: 
\begin{enumerate}
\item vanishing of cohomology $\Rightarrow$ parameter rigidity. \\
This is the sufficient condition for parameter rigidity proved in \cite{Ma2}. In the current article this is Theorem \ref{SC2}\label{1}
\item cohomology vanishing results. \\
These are by Katok--Spatzier \cite{KS}, \cite{KS2} and Kanai \cite{Ka}. See Theorem \ref{vani} and Corollary \ref{im} in this paper\label{2}
\item bridging the gap between Step \ref{1} and Step \ref{2}. \\
This is because the cohomology vanishing results are available only for finitely many coefficients, while the sufficient condition for parameter rigidity requires vanishing of cohomology for seemingly much more coefficients. Here we use Proposition \ref{tt}, which shows the relevance to large scale geometry. Then the main point is that our acting group $AN$ is isometric to $G/K$ by Iwasawa decomposition $G=ANK$. So we can use the rigidity theorems of Pansu \cite{Pa} and Kleiner--Leeb \cite{KL} on quasiisometries of symmetric spaces, and a certain rigidity property of quasiisometries of hyperbolic spaces proved in Farb--Mosher \cite{FM} and Reiter Ahlin \cite{RA}. 
\end{enumerate}

Theorem \ref{at least 2} shows a contrast between the higher rank case and $\widetilde{\PSL}(2,\bb{R})$, the universal cover of $\PSL(2,\bb{R})$, for which Asaoka \cite{A2} gives (generally) nontrivial orbit-preserving deformations of the actions of $AN$ by right multiplication.

\begin{thm}[Asaoka \cite{A2}]
Let $\Gamma$ be a cocompact lattice in $\widetilde{\PSL}(2,\bb{R})$ and let 
\begin{equation*}
A=\left\{
\begin{pmatrix}
a&\\
&a^{-1}
\end{pmatrix}
\middle|a>0\right\}, \quad
N=\left\{
\begin{pmatrix}
1&b\\
&1
\end{pmatrix}
\middle|b\in\bb{R}\right\}. 
\end{equation*}
Let $\Phi_\Gamma$ be the flow on $\Gamma\backslash\widetilde{\PSL}(2,\bb{R})$ defined by the action of $A$ by right multiplication, $\mca{P}$ be the set of oriented periodic orbits of $\Phi_\Gamma$ and $\tau(\gamma)$ be the period of $\gamma$ for $\gamma\in\mca{P}$. Consider 
\begin{equation*}
\Delta_\Gamma=\left\{a\in H^1\left(\Gamma\backslash\widetilde{\PSL}(2,\bb{R});\bb{R}\right)\ \middle| \ \sup_{\gamma\in\mca{P}}\frac{\vert a(\gamma)\rvert}{\tau(\gamma)}<1\right\}
\end{equation*}
which is an open neighborhood of $0$ in $H^1\left(\Gamma\backslash\widetilde{\PSL}(2,\bb{R});\bb{R}\right)$. Then there exists an analytic locally free action $\rho_a$ of $AN$ on $\Gamma\backslash\widetilde{\PSL}(2,\bb{R})$ for each $a\in\Delta_\Gamma$ with the following properties: 
\begin{itemize}
\item The action $\rho_0$ is defined by the right multiplication. 
\item All the $\rho_a$'s have the same orbit foliation $\mca{F}$. 
\item Actions $\rho_a$ and $\rho_{a^\prime}$ are not parameter equivalent if $a\neq a^\prime$. 
\item Any $C^\infty$ locally free action of $AN$ whose orbit foliation is $\mca{F}$ is parameter equivalent to $\rho_a$ for some $a\in\Delta_\Gamma$. 
\item The action $\rho_a$ does not preserve any $C^0$ volume form on $\Gamma\backslash\widetilde{\PSL}(2,\bb{R})$ except when $a=0$. 
\end{itemize}
\end{thm}

We also know how the action $\rho_a$ is controlled by the cohomology class $a$, but we refer the reader to \cite{A2} for that and more information. Note that the above deformation is different from the nonorbit-preserving deformation coming from the deformation of the lattice, whose deformation space has the dimension equal to that of Teichm\"{u}ller space, because such deformations are necessarily $C^0$ volume preserving.

\section{Preliminaries}\label{prelimi}
This section is a summary of the results we need later, proved in Maruhashi \cite{Ma2}. See \cite{Ma2} for the detail. In this paper Lie algebras are denoted by the corresponding lowercase Fraktur of the corresponding Lie groups. The symbol $\Gamma(\ \cdot\ )$ denotes the set of all $C^\infty$ sections of a vector bundle.

\subsection{Leafwise cohomology}
Let $M\stackrel{\rho_0}{\curvearrowleft}S$ be a $C^\infty$ locally free action of a connected simply connected solvable Lie group $S$ on a closed manifold $M$ with the orbit foliation $\mca{F}$. Let $\omega_0\in\Gamma\left(\Hom(T\mca{F},\mf{s})\right)$ denote {\em the canonical $1$-form of $\rho_0$}, ie $(\omega_0)_x\colon T_x\mca{F}\to\mf{s}$ for $x\in M$ is defined as the inverse of the derivative at the identity of the map $S\to M$, $s\mapsto\rho_0(x,s)$. Let 
\begin{equation*}
\df\colon\Gamma\left(\bigwedge^pT^*\mca{F}\right)\to\Gamma\left(\bigwedge^{p+1}T^*\mca{F}\right)
\end{equation*}
be the leafwise exterior derivative of $\mca{F}$, defined by the same formula as the usual exterior derivative. Then $\omega_0$ satisfies the Maurer--Cartan equation $d_\mca{F}\omega_0+[\omega_0,\omega_0]=0$. Here $\df\omega_0$ and $[\omega_0,\omega_0]$ are defined by 
\begin{equation*}
\df\omega_0(X,Y)=X\omega_0(Y)-Y\omega_0(X)-\omega_0\left([X,Y]\right)
\end{equation*}
and 
\begin{equation*}
[\omega_0,\omega_0](X,Y)=\left[\omega_0(X),\omega_0(Y)\right]
\end{equation*}
for $X$, $Y\in\Gamma(T\mca{F})$. Let $\mf{s}\stackrel{\pi}{\curvearrowright}V$ be a representation of $\mf{s}$ on a finite dimensional real vector space $V$. Then $\pi\omega_0\in\Gamma\left(\Hom\left(T\mca{F},\End(V)\right)\right)$ satisfies 
\begin{equation*}
d_\mca{F}\pi\omega_0+\left[\pi\omega_0,\pi\omega_0\right]=0. 
\end{equation*}
We regard $\pi\omega_0$ as the connection form of a flat $\mca{F}$-partial connection $\nabla$ of the trivial vector bundle $M\times V\to M$ relative to any global frame of the bundle which has constant $V$ components, ie $\nabla_Xv=\pi\left(\omega_0(X)\right)v$ for $X\in\Gamma(T\mca{F})$ and $v\in V$, where $v$ is regarded as a section of $M\times V\to M$. Hence $\nabla\xi=\df\xi+\pi\omega_0\xi$ for general $\xi\in\Gamma(V)$. The exterior derivative of $\nabla$ is 
\begin{align*}
\Gamma\left(\bigwedge^pT^*\mca{F}\otimes V\right)&\to\Gamma\left(\bigwedge^{p+1}T^*\mca{F}\otimes V\right)\\
\omega&\mapsto\df\omega+\pi\omega_0\wedge\omega, 
\end{align*}
where our definition of exterior product is 
\begin{equation*}
\left(\pi\omega_0\wedge\omega\right)\left(X_1,\ldots,X_{p+1}\right)=\sum_{i=1}^{p+1}(-1)^{i+1}\pi\omega_0(X_i)\omega\left(X_1,\ldots,\widehat{X_i},\ldots,X_{p+1}\right). 
\end{equation*}
The square of this operator is zero by the flatness. The cohomology $H^*\left(\mca{F};\mf{s}\stackrel{\pi}{\curvearrowright}V\right)$ of this complex is {\em the leafwise cohomology of $\mca{F}$ with coefficient $\pi$}. Recall that the cohomology $H^*\left(\mf{s};\mf{s}\stackrel{\pi}{\curvearrowright}V\right)$ of the Lie algebra $\mf{s}$ with coefficient $\pi$ is obtained from the complex $\Hom\left(\bigwedge^*\mf{s},V\right)$. We have an injective cochain map 
\begin{align*}
\Hom\left(\bigwedge^*\mf{s},V\right)&\hookrightarrow\Gamma\left(\bigwedge^*T^*\mca{F}\otimes V\right)\\
\varphi&\mapsto\omega_0^*\varphi, 
\end{align*}
where $\omega_0^*$ is the pullback by $\omega_0$. Then by Lemma 2.1.3 of \cite{Ma2}, the induced map 
\begin{equation*}
H^*\left(\mf{s};\mf{s}\stackrel{\pi}{\curvearrowright}V\right)\to H^*\left(\mca{F};\mf{s}\stackrel{\pi}{\curvearrowright}V\right)
\end{equation*}
is injective and we see $H^*\left(\mf{s};\mf{s}\stackrel{\pi}{\curvearrowright}V\right)$ as a subspace of $H^*\left(\mca{F};\mf{s}\stackrel{\pi}{\curvearrowright}V\right)$.

\subsection{A sufficient condition for parameter rigidity}
Let $\mf{n}$ denote the nilradical of $\mf{s}$. We have $[\mf{s},\mf{s}]\subset\mf{n}$. Take a subspace $\mf{h}$ such that $[\mf{s},\mf{s}]\subset\mf{h}\subset\mf{n}$. Then $\mf{h}$ is a nilpotent ideal of $\mf{s}$ and let 
\begin{equation*}
\mf{h}\supset\mf{h}^2\supset\cdots\supset\mf{h}^d\supset 0
\end{equation*}
be the lower central series of $\mf{h}$. This filtration of $\mf{h}$ is invariant with respect to $\mf{s}\stackrel{\ad}{\curvearrowright}\mf{h}$. Let 
\begin{equation*}
\mf{s}\stackrel{\ad}{\curvearrowright}\Gr(\mf{h})=\bigoplus_{i=1}^d\mf{h}^i/\mf{h}^{i+1}
\end{equation*}
be the associated graded quotient. Since $\mf{h}$ acts trivially, we get $\mf{s}/\mf{h}\stackrel{\ad}{\curvearrowright}\Gr(\mf{h})$. 

Let $\mca{A}(\mca{F},S)$ be the set of all $C^\infty$ locally free actions $M\curvearrowleft S$ with the orbit foliation $\mca{F}$. Let $\rho\in\mca{A}(\mca{F},S)$, and let $\omega$ denote the canonical $1$-form of $\rho$. Let $p\colon\mf{s}\to\mf{s}/\mf{h}$ denote the natural projection. Applying $p$ to $\df\omega+[\omega,\omega]=0$, we get $\df p\omega=0$. Assume $H^1(\mca{F})=H^1(\mf{s})$. Then $[p\omega]\in H^1\left(\mca{F};\mf{s}/\mf{h}\right)=H^1\left(\mf{s};\mf{s}/\mf{h}\right)$. So there exist a unique linear map $\varphi_\rho\colon\mf{s}\to\mf{s}/\mf{h}$ which vanishes on $[\mf{s},\mf{s}]$ and a $C^\infty$ map $h\colon M\to\mf{s}/\mf{h}$ such that 
\begin{equation*}
p\omega=\varphi_\rho\omega_0+\df h. 
\end{equation*}
The map $\varphi_\rho$ is surjective by Lemma 2.2.2 of \cite{Ma2}. 

\begin{thm}[Maruhashi \cite{Ma2}]\label{SC2}
If 
\begin{equation*}
H^1(\mca{F})=H^1(\mf{s})
\end{equation*}
and 
\begin{equation*}
H^1\left(\mca{F};\mf{s}\stackrel{\ad\circ\varphi_\rho}{\curvearrowright}\Gr(\mf{h})\right)=H^1\left(\mf{s};\mf{s}\stackrel{\ad\circ\varphi_\rho}{\curvearrowright}\Gr(\mf{h})\right)
\end{equation*}
for some $\mf{h}$ and for all $\rho\in\mca{A}(\mca{F},S)$, then $M\stackrel{\rho_0}{\curvearrowleft}S$ is parameter rigid. 
\end{thm}

See Theorem 2.2.5 of \cite{Ma2} for this theorem.

\subsection{A property from large scale geometry}
Let $\rho\in\mca{A}(\mca{F},S)$ and let $a_\rho\colon M\times S\to S$ be the unique $C^\infty$ map satisfying 
\begin{equation*}
\rho_0(x,s)=\rho\left(x,a_\rho(x,s)\right)\quad\text{and}\quad a_\rho(x,1)=1
\end{equation*}
for all $x\in M$ and $s\in S$. The map $a_\rho$ is defined since $\rho_0$ and $\rho$ have the same orbit foliation. It is known that $a_\rho$ is a cocycle over $\rho_0$. 

Let $X$, $B$ be metric spaces. A surjective map $p\colon X\to B$ is {\em a distance respecting projection} if 
\begin{equation*}
d(b,b^\prime)=d\left(p^{-1}(b),p^{-1}(b^\prime)\right)=d_\mca{H}\left(p^{-1}(b),p^{-1}(b^\prime)\right)
\end{equation*}
holds for all $b$, $b^\prime\in B$, where 
\begin{equation*}
d\left(p^{-1}(b),p^{-1}(b^\prime)\right)=\inf\left\{d(x,x^\prime)\ \middle|\ x\in p^{-1}(b),x^\prime\in p^{-1}(b^\prime)\right\}
\end{equation*}
and $d_\mca{H}$ denotes the Hausdorff distance. Let $p\colon X\to B$ and $p^\prime\colon X^\prime\to B^\prime$ be distance respecting projections. A diagram 
\begin{equation*}
\begin{tikzcd}
X\ar[r,"f"]\ar[d,"p"']&X^\prime\ar[d,"p^\prime"]\\
B\ar[r,"\varphi"']&B^\prime
\end{tikzcd}
\end{equation*}
is {\em fiber respecting} or {\em $f$ is fiber respecting over $\varphi$} if $f$ and $\varphi$ are maps and there exists a constant $C>0$ such that $d_\mca{H}\left(f\left(p^{-1}(b)\right),(p^\prime)^{-1}\left(\varphi(b)\right)\right)<C$ for all $b\in B$. 

\begin{prop}\label{metric}
Let $G$ be a connected Lie group and $H$ a connected normal closed subgroup of $G$. Take an inner product of $\mf{g}$. Endow $\mf{g}/\mf{h}$ with the inner product for which the restriction $\mf{h}^\perp\stackrel{\sim}{\to}\mf{g}/\mf{h}$ of the projection $\mf{g}\to\mf{g}/\mf{h}$ is an isometry. Give $G$ and $G/H$ left invariant Riemannian metrics corresponding to these inner products. Then the projection $p\colon G\to G/H$ is a distance respecting projection. 
\end{prop}

\begin{proof}
This follows from Lemma 4.1.1 of \cite{Ma2} by noting that $H\stackrel{\Ad}{\curvearrowright}\mf{g}/\mf{h}$ is trivial. 
\end{proof}

Assume $H^1(\mca{F})=H^1(\mf{s})$ for an action $M\stackrel{\rho_0}{\curvearrowleft}S$ and let $\rho\in\mca{A}(\mca{F},S)$ and $\varphi_\rho\colon\mf{s}\to\mf{s}/\mf{h}$, $a_\rho\colon M\times S\to S$ as above. Let $K_\rho$ and $H$ be the Lie subgroups corresponding to $\ker\varphi_\rho$ and $\mf{h}$. Then $S/K_\rho$ and $S/H$ are vector groups. Let $\tilde{\varphi}_\rho\colon S/K_\rho\to S/H$ be the linear isomorphism with differential $\varphi_\rho\colon\mf{s}/\ker\varphi_\rho\simeq\mf{s}/\mf{h}$. 

\begin{prop}[Maruhashi \cite{Ma2}]\label{tt}
For any $\rho\in\mca{A}(\mca{F},S)$, $x\in M$ and $\mf{h}$, consider the diagram 
\begin{equation*}
\begin{tikzcd}
S\ar[r,"{a_\rho(x,\cdotp)}"]\ar[d]&S\ar[d]\\
S/K_\rho\ar[r,"\tilde{\varphi}_\rho"']&S/H, 
\end{tikzcd}
\end{equation*}
where the vertical maps are the natural projections. Fix an inner product of $\mf{s}$ and give $S$, $S/K_\rho$ and $S/H$ left invariant Riemannian metrics considered in Proposition \ref{metric}. Then $a_\rho(x,\cdotp)$ is a fiber respecting biLipschitz diffeomorphism over $\tilde{\varphi}_\rho$. (In particular $a_\rho(x,\cdotp)$ is a quasiisometry.)
\end{prop}

See Proposition 4.1.4 of \cite{Ma2} for this proposition.

\section{Reduction of the proof of Theorem \ref{at least 2} to Proposition \ref{qq}}\label{an1}
Let $G$ be a connected semisimple Lie group. Fix a Cartan decomposition $\mf{g}=\mf{k}\oplus\mf{p}$ and a maximal abelian subspace $\mf{a}$ of $\mf{p}$. Let $\Sigma$ be the restricted root system of $(\mf{g},\mf{a})$ and fix a positive system $\Sigma_+$ of $\Sigma$. Let $\mf{n}=\bigoplus_{\lambda\in\Sigma_+}\mf{g}_\lambda$, where 
\begin{equation*}
\mf{g}_\lambda=\lbrace X\in\mf{g}\mid[H,X]=\lambda(H)X\text{ for all $H\in\mf{a}$}\rbrace
\end{equation*}
is a restricted root space. Let $K$, $A$ and $N$ be the Lie subgroups corresponding to $\mf{k}$, $\mf{a}$ and $\mf{n}$. Then $G=KAN$ is an Iwasawa decomposition. The group $AN$ is a connected simply connected solvable Lie group and its Lie algebra is $\mf{an}=\mf{n}\rtimes\mf{a}$. 

It is easy to show that $\mf{n}$ is the nilradical of $\mf{an}$ and $\mf{n}=[\mf{an},\mf{an}]$. So we must take $\mf{h}=\mf{n}$ to apply Theorem \ref{SC2}. Then 
\begin{equation*}
\mf{an}/\mf{n}=\mf{a}\stackrel{\ad}{\curvearrowright}\Gr(\mf{n})=\bigoplus_{i\geq1}\mf{n}^i/\mf{n}^{i+1}
\end{equation*}
is isomorphic to 
\begin{equation*}
\mf{a}\stackrel{\ad}{\curvearrowright}\mf{n}=\bigoplus_{\lambda\in\Sigma_+}\mf{g}_\lambda.
\end{equation*}
To apply Theorem \ref{SC2}, we must show $H^1(\mca{F})=H^1(\mf{an})$ and then calculate cohomology with coefficient 
\begin{equation}\label{dec}
\mf{an}\stackrel{\ad\circ\varphi_\rho}{\curvearrowright}\Gr(\mf{n}),\ \ \ \ \text{ie}\ \ \ \ \mf{an}\stackrel{\ad\circ\varphi_\rho}{\curvearrowright}\mf{n}=\bigoplus_{\lambda\in\Sigma_+}\mf{g}_\lambda
\end{equation}
for any $\rho\in\mca{A}(\mca{F},AN)$, where $\varphi_\rho\colon\mf{an}\to\mf{a}$. Note that $\ker\varphi_\rho=\mf{n}$ and $\varphi_\rho|_\mf{a}\in\GL(\mf{a})$. The $\mf{g}_\lambda$-component $\mf{an}\stackrel{\ad\circ\varphi_\rho}{\curvearrowright}\mf{g}_\lambda$ in \eqref{dec} is a direct sum of the $1$-dimensional representation $\mf{an}\stackrel{\lambda\circ\varphi_\rho}{\curvearrowright}\bb{R}$. Therefore, we get the following. 

\begin{lem}\label{yy}
Let $G$ be a connected semisimple Lie group. Fix a Cartan decomposition $\mf{g}=\mf{k}\oplus\mf{p}$, a maximal abelian subspace $\mf{a}$ of $\mf{p}$ with the associated restricted root system $\Sigma$ and a positive system $\Sigma_+$ of $\Sigma$. Let $\mca{F}$ be the orbit foliation of a $C^\infty$ locally free action $M\stackrel{\rho_0}{\curvearrowleft}AN$ on a closed manifold $M$. If 
\begin{equation}\label{trivial}
H^1(\mca{F})=H^1(\mf{an})
\end{equation}
and 
\begin{equation}\label{nontrivial}
H^1\left(\mca{F};\mf{an}\stackrel{\lambda\circ\varphi_\rho}{\curvearrowright}\bb{R}\right)=H^1\left(\mf{an};\mf{an}\stackrel{\lambda\circ\varphi_\rho}{\curvearrowright}\bb{R}\right)
\end{equation}
for any $\lambda\in\Sigma_+$ and $\rho\in\mca{A}(\mca{F},AN)$, 
then $\rho_0$ is parameter rigid. 
\end{lem}

Before proving Theorem \ref{at least 2} we remark that the same result but with a stronger assumption of real rank at least $3$ follows easily from the following result of Kononenko \cite[Theorem 8.2]{K2}. 

\begin{thm}[Kononenko \cite{K2}]\label{nontri}
Let $G$ be a connected semisimple Lie group with finite center of real rank at least $3$ whose simple factors are of real rank at least $2$, and let $\Gamma$ be an irreducible cocompact lattice in $G$. Let $G=KAN$ be an Iwasawa decomposition. Take $\mu\colon\mf{an}\to\bb{R}$ be any nonzero linear function which vanishes on $\mf{n}$, and let $\tilde{\mu}\colon AN\to\GL(1,\bb{R})$ be the homomorphism with differential $\mu$. Then any $\tilde{\mu}$-twisted $C^\infty$ cocycle over the action $\Gamma\backslash G\curvearrowleft AN$ by right multiplication is $C^\infty$ cohomologous to a constant cocycle. Equivalently we have 
\begin{equation*}
H^1\left(\mca{F};\mf{an}\stackrel{\mu}{\curvearrowright}\bb{R}\right)=H^1\left(\mf{an};\mf{an}\stackrel{\mu}{\curvearrowright}\bb{R}\right). 
\end{equation*}
\end{thm}

\begin{cor}
Let $G$ be a connected semisimple Lie group with finite center of real rank at least $3$ whose simple factors are of real rank at least $2$, and let $\Gamma$ be an irreducible cocompact lattice in $G$. Let $G=KAN$ be an Iwasawa decomposition. Then the action $\Gamma\backslash G\curvearrowleft AN$ by right multiplication is parameter rigid. 
\end{cor}

\begin{proof}
Under the assumptions of this corollary, \eqref{trivial} in Lemma \ref{yy} follows from the case $\lambda=0$ in Corollary \ref{im} below and \eqref{nontrivial} in Lemma \ref{yy} follows from Theorem \ref{nontri}, which imply parameter rigidity of the action. 
\end{proof}

But this does not cover the case of real rank $2$. In this case we only know vanishing of the cohomology with coefficients corresponding to restricted roots. 

\begin{thm}\label{vani}
Let $G$ be a connected semisimple Lie group with finite center of real rank at least $2$ without compact factors or simple factors locally isomorphic to $\SO_0(n,1)$ $(n\geq2)$ or $\SU(n,1)$ $(n\geq2)$, and let $\Gamma$ be an irreducible cocompact lattice in $G$. Fix a Cartan decomposition $\mf{g}=\mf{k}\oplus\mf{p}$ and a maximal abelian subspace $\mf{a}$ of $\mf{p}$ with the associated restricted root system $\Sigma$. Let $\mca{F}_A$ be the orbit foliation of the action $\Gamma\backslash G\curvearrowleft A$ by right multiplication. Then we have: 
\begin{enumerate}
\item {\rm (Katok--Spatzier \cite[Theorem 3.6]{KS2})}
\begin{equation*}
H^1(\mca{F}_A)=H^1(\mf{a})
\end{equation*}
\item {\rm (Kanai \cite[Theorem 2.2]{Ka})}
\begin{equation*}
H^1\left(\mca{F}_A;\mf{a}\stackrel{\lambda}{\curvearrowright}\bb{R}\right)=H^1\left(\mf{a};\mf{a}\stackrel{\lambda}{\curvearrowright}\bb{R}\right)=0
\end{equation*}
for any $\lambda\in\Sigma$. 
\end{enumerate}
\end{thm}

\begin{rem}
In Theorem 2.2 (2) of \cite{Ka} it is written that $u$ (a notation from \cite{Ka}) is $C^\infty$ if the conditions (i) and (ii) from that paper are satisfied, but those conditions (i) and (ii) are {\em always} satisfied, so that we get the above result. 
\end{rem}

\begin{cor}\label{im}
Let $G$ be a connected semisimple Lie group with finite center of real rank at least $2$ without compact factors or simple factors locally isomorphic to $\SO_0(n,1)$ $(n\geq2)$ or $\SU(n,1)$ $(n\geq2)$, and let $\Gamma$ be an irreducible cocompact lattice in $G$. Fix a Cartan decomposition $\mf{g}=\mf{k}\oplus\mf{p}$, a maximal abelian subspace $\mf{a}$ of $\mf{p}$ with the associated restricted root system $\Sigma$ and a positive system $\Sigma_+$ of $\Sigma$. Let $\mca{F}$ be the orbit foliation of the action $\Gamma\backslash G\curvearrowleft AN$ by right multiplication. Then we have 
\begin{equation*}
H^1\left(\mca{F};\mf{an}\stackrel{\lambda}{\curvearrowright}\bb{R}\right)=H^1\left(\mf{an};\mf{an}\stackrel{\lambda}{\curvearrowright}\bb{R}\right)
\end{equation*}
for any $\lambda\in\Sigma\cup\lbrace0\rbrace$, where $\lambda\colon\mf{a}\to\bb{R}$ is regarded as $\lambda\colon\mf{an}\to\bb{R}$ by extending it as $0$ on $\mf{n}$. 
\end{cor}

\begin{proof}
Let $[\omega]\in H^1\left(\mca{F};\mf{an}\stackrel{\lambda}{\curvearrowright}\bb{R}\right)$, that is, $\df\omega+\lambda\omega_0\wedge\omega=0$, where $\omega_0$ is the canonical $1$-form of $\Gamma\backslash G\curvearrowleft AN$. By restriction to $T\mca{F}_A$ we get $d_{\mca{F}_A}\omega+\lambda\omega_0\wedge\omega=0$. Note that $\omega_0$ restricts to the canonical $1$-form of $\Gamma\backslash G\curvearrowleft A$. So 
\begin{equation*}
[\omega]\in H^1\left(\mca{F}_A;\mf{a}\stackrel{\lambda}{\curvearrowright}\bb{R}\right)=H^1\left(\mf{a};\mf{a}\stackrel{\lambda}{\curvearrowright}\bb{R}\right). 
\end{equation*}
There exist a linear map $\phi\colon\mf{a}\to\bb{R}$ such that $\lambda(H)\phi(H^\prime)-\lambda(H^\prime)\phi(H)=0$ for all $H$, $H^\prime\in\mf{a}$ and a $C^\infty$ function $h\colon\Gamma\backslash G\to\bb{R}$ satisfying $\omega=\phi\omega_0+d_{\mca{F}_A}h+\lambda\omega_0h$. For any $H\in\mf{a}$ and $X\in\mf{g}_\mu$ for $\mu\in\Sigma_+$, 
\begin{align*}
0&=H\omega(X)-X\omega(H)-\mu(H)\omega(X)+\lambda(H)\omega(X)\\ 
&=H\omega(X)-X(\phi(H)+Hh+\lambda(H)h)-\mu(H)\omega(X)+\lambda(H)\omega(X)\\ 
&=H\omega(X)-HXh-[X,H]h-\lambda(H)Xh-\mu(H)\omega(X)+\lambda(H)\omega(X)\\ 
&=H(\omega(X)-Xh)+(\lambda(H)-\mu(H))(\omega(X)-Xh). 
\end{align*}
If $\mu\neq\lambda$, take $H_0\in\mf{a}$ such that $\lambda(H_0)-\mu(H_0)\neq0$. Then the above equation for $H=H_0$ and the boundedness of $\omega(X)-Xh$ imply $\omega(X)-Xh=0$. If $\mu=\lambda$, take $H\neq0$. We can apply Moore's Ergodicity Theorem since $G$ has finite center and no compact factor and $\Gamma$ is irreducible. So the flow $e^{tH}$ $(t\in\bb{R})$ has a dense orbit and $\omega(X)-Xh=\psi(X)$ for some $\psi(X)\in\bb{R}$. Let $\omega^\prime=\omega-\df h-\lambda\omega_0h$. Then 
\begin{align*}
\omega^\prime(H)&=\phi(H)\quad\text{for}\ H\in\mf{a}, \\
\omega^\prime(X)&=
\begin{cases}
0&\text{for}\ X\in\mf{g}_\mu\ \text{and}\ \mu\neq\lambda\\
\psi(X)&\text{for}\ X\in\mf{g}_\lambda\ \text{and}\ \lambda\in\Sigma_+. 
\end{cases}
\end{align*}
Therefore, $[\omega]=[\omega^\prime]\in H^1\left(\mf{an};\mf{an}\stackrel{\lambda}{\curvearrowright}\bb{R}\right)$. 
\end{proof}

By Corollary \ref{im}, the proof of Theorem \ref{at least 2} reduces to the following proposition. 

\begin{prop}\label{qq}
Let $G$ be a connected semisimple Lie group. Fix a Cartan decomposition $\mf{g}=\mf{k}\oplus\mf{p}$, a maximal abelian subspace $\mf{a}$ of $\mf{p}$ with the associated restricted root system $\Sigma$ and a positive system $\Sigma_+$ of $\Sigma$. Let $G=KAN$ be the corresponding Iwasawa decomposition. Let $M\stackrel{\rho_0}{\curvearrowleft}AN$ be a $C^\infty$ locally free action on a closed manifold $M$ with the orbit foliation $\mca{F}$. If 
\begin{equation*}
H^1(\mca{F})=H^1(\mf{an})
\end{equation*}
and 
\begin{equation*}
H^1\left(\mca{F};\mf{an}\stackrel{\lambda}{\curvearrowright}\bb{R}\right)=H^1\left(\mf{an};\mf{an}\stackrel{\lambda}{\curvearrowright}\bb{R}\right)
\end{equation*}
for all $\lambda\in\Sigma_+$, where $\lambda\colon\mf{a}\to\bb{R}$ is regarded as $\lambda\colon\mf{an}\to\bb{R}$ by extending it linearly as $0$ on $\mf{n}$, then $\rho_0$ is parameter rigid. 
\end{prop}

Note that we need no assumption on the simple factors of $G$ in this proposition. 

\begin{rem}
In Theorem \ref{vani} we assume that \vspace{5pt}

\noindent
$(\ast)$ $G$ has no simple factors locally isomorphic to $\SO_0(n,1)$ $(n\geq2)$ or $\SU(n,1)$ $(n\geq2)$. \vspace{5pt}

\noindent
If $1$ in Theorem \ref{vani} is true without the assumption $(\ast)$, then $2$ in Theorem \ref{vani} and Corollary \ref{im} are true without the assumption $(\ast)$. Hence Theorem \ref{at least 2} will be true without the assumption $(\ast)$ by Proposition \ref{qq}. 
\end{rem}

\section{Proof of Proposition \ref{qq}}\label{an2}
To prove Proposition \ref{qq}, it suffices to show that $\lambda\circ\varphi_\rho|_\mf{a}\in\Sigma_+$ for any $\lambda\in\Sigma_+$ and any $\rho\in\mca{A}(\mca{F},AN)$ by Lemma \ref{yy}. At this moment we know $\varphi_\rho|_\mf{a}$ is only an element of $\GL(\mf{a})$, so it is not clear whether $\varphi_\rho|_\mf{a}$ preserves $\Sigma_+$. To prove it we need rigidity of quasiisometries of symmetric spaces. 

For the proof of Proposition \ref{qq} we may assume that $G$ has no compact factors, since this does not change $AN$. Recall that $\Inn(\mf{g})=\Ad(G)=G/Z(G)$, where $Z(G)$ denotes the center of $G$, and $G/Z(G)$ has the trivial center. Replacing $G$ with $G/Z(G)$ also does not change $AN$, so we may assume $G=\Inn(\mf{g})$ as well. 

The mapping $an\mapsto anK$ gives a canonical diffeomorphism $AN\simeq G/K$ by the Iwasawa decomposition. Henceforth we identify $AN$ with $G/K$ in this way. This is $AN$ equivariant. 

Recall that the identification $\mf{p}\simeq T_KG/K$ is by $X\mapsto\left.\frac{d}{dt}e^{tX}K\right\vert_{t=0}$. In the following $K$ denotes the subgroup $K$ or the point $K$ in $G/K$ depending on the context. $G$-invariant Riemannian metrics on $G/K$ are in one-to-one correspondence with inner products on $\mf{p}$ invariant under $K\stackrel{\Ad}{\curvearrowright}\mf{p}$. We equip $G/K$ with a $G$-invariant Riemannian metric $g$ corresponding to the restriction of $B_\theta$ to $\mf{p}$, where $\theta$ is the Cartan involution associated with the Cartan decomposition $\mf{g}=\mf{k}\oplus\mf{p}$, $B$ the Killing form of $\mf{g}$ and $B_\theta(X,Y)=-B(X,\theta Y)$ for $X,Y\in\mf{g}$. The restriction of $B_\theta$ to $\mf{p}$ is the same as the restriction of $B$ to $\mf{p}$. We give $AN$ the Riemannian metric which makes the identification $AN\simeq G/K$ an isometry. This Riemannian metric is $AN$ invariant. Geodesics in $G/K$ passing $K$ at time $0$ are of the form $e^{tX}K$ $(t\in\bb{R})$ for $X\in\mf{p}$. Note that $e^{tX}K$ $(t\in\bb{R})$ for $X\in\mf{g}\setminus\mf{p}$ is not a geodesic in general. In $AN$ curves of the form $ne^{tH}$ $(t\in\bb{R})$ for fixed $n\in N$ and $H\in\mf{a}$ are geodesics. 

The decomposition $\mf{g}=\mf{k}\oplus\mf{p}$ is orthogonal with respect to the positive definite symmetric bilinear form $B_\theta$. Let $\mf{g}_\lambda^\prime$ be the orthogonal projection to $\mf{p}$ with respect to $\mf{g}=\mf{k}\oplus\mf{p}$ of $\mf{g}_\lambda$ for $\lambda\in\Sigma$. The space $\mf{g}_\lambda^\prime$ has the same dimension as $\mf{g}_\lambda$ since $\mf{k}=\ker(\theta-\id)$ and $\theta\mf{g}_\lambda=\mf{g}_{-\lambda}$. This orthogonal projection maps $\mf{an}$ isomorphically to $\mf{p}$ by the Iwasawa decomposition $\mf{g}=\mf{k}\oplus\mf{a}\oplus\mf{n}$. Therefore, 
\begin{equation*}
\mf{p}=\mf{a}\oplus\mf{n}^\prime, \ \ \text{where} \ \ \mf{n}^\prime=\bigoplus_{\lambda\in\Sigma_+}\mf{g}_\lambda^\prime. 
\end{equation*}
Note that $\mf{a}\perp\mf{n}^\prime$ since $\mf{a}\perp\mf{g}_\lambda$ for $\lambda\in\Sigma$ and $\mf{a}\perp\mf{k}$ with respect to $B_\theta$. Observe that the differentiation 
\begin{equation*}
\begin{tikzcd}
\mf{an}\ar[r,"\sim"]\ar[d,equal]\ar[dr,phantom,"\circlearrowright" marking]&\mf{p}=\mf{a}\oplus\mf{n}^\prime\ar[d,equal]\\
T_1AN\ar[r,"\sim"]&T_KG/K
\end{tikzcd}
\end{equation*}
at $1$ of the identification $AN\simeq G/K$ maps $\mf{an}$ to $\mf{p}$ by the orthogonal projection with respect to $\mf{g}=\mf{k}\oplus\mf{p}$. Therefore, $\mf{a}$ maps identically to $\mf{a}$ and $\mf{n}$ maps isomorphically to $\mf{n}^\prime$. So $\mf{a}\perp\mf{n}$ in $\mf{an}$. 

For any $\rho\in\mca{A}(\mca{F},AN)$ and $x\in M$, consider the diagram 
\begin{equation*}
\begin{tikzcd}
AN\ar[r,"{a_\rho(x,\cdotp)}"]\ar[d,"p"']&AN\ar[d,"p"]\\
A\ar[r,"\tilde{\varphi}_\rho"',"\sim"]&A, 
\end{tikzcd}
\end{equation*}
where $p$ is the natural projection. We give $A$ a left invariant Riemannian metric for which the restriction $\mf{a}\to\mf{a}$ of the natural projection $\mf{an}\to\mf{a}$ to $\mf{n}^\perp=\mf{a}$ becomes an isometry, ie we consider the restriction of $B$ to $\mf{a}$. Then $p$ is a distance respecting projection by Proposition \ref{metric} and $a_\rho(x,\cdotp)$ is a fiber respecting biLipschitz diffeomorphism over $\tilde{\varphi}_\rho$ by Proposition \ref{tt}. 

Since $G=\Ad(G)$, we have $G=G_1\times\cdots\times G_\ell$, where $G_i$ is a connected noncompact simple Lie group with trivial center. Since any two maximal compact subgroups of $G$ are conjugate by an inner automorphism of $G$, we have $K=K_1\times\cdots\times K_\ell$, where $K_i$ is a maximal compact subgroup of $G_i$ and $G/K=G_1/K_1\times\cdots\times G_\ell/K_\ell$. Let $\mf{g}_i=\mf{k}_i\oplus\mf{p}_i$ be the Cartan decomposition. Then $\mf{p}=\mf{p}_1\oplus\cdots\oplus\mf{p}_\ell$. Let $g_i$ be the $G_i$-invariant Riemannian metric on $G_i/K_i$ corresponding to the restriction of the Killing form $B_i$ of $\mf{g}_i$ to $\mf{p}_i$. Since 
\begin{equation*}
B\left((X_1,\ldots,X_\ell),(Y_1,\ldots,Y_\ell)\right)=B_1(X_1,Y_1)+\cdots+B_\ell(X_\ell,Y_\ell)
\end{equation*}
for $X_i$, $Y_i\in\mf{g}_i$, we have $g=g_1\times\cdots\times g_\ell$. Since maximal abelian subspaces in $\mf{p}$ are conjugate by $\Ad(k)$ for some $k\in K$ and $\Ad(k)$ preserves each $\mf{p}_i$, we have $\mf{a}=\mf{a}_1\oplus\cdots\oplus\mf{a}_\ell$ for some maximal abelian subspace $\mf{a}_i$ of $\mf{p}_i$. Let 
\begin{equation*}
\mf{g}_i=\mf{a}_i\oplus\mf{m}_i\oplus\bigoplus_{\lambda_i\in\Sigma_i}(\mf{g}_i)_{\lambda_i}
\end{equation*}
be the restricted root space decomposition of $\mf{g}_i$. Then 
\begin{equation*}
\mf{g}=\bigoplus_{i=1}^\ell\mf{a}_i\oplus\bigoplus_{i=1}^\ell\mf{m}_i\oplus\bigoplus_{i=1}^\ell\bigoplus_{\lambda_i\in\Sigma_i}(\mf{g}_i)_{\lambda_i}
\end{equation*}
is the restricted root space decomposition of $\mf{g}$. Thus $\Sigma=\Sigma_1\cup\cdots\cup\Sigma_\ell$, where $\lambda_i\colon\mf{a}_i\to\bb{R}$ in $\Sigma_i$ is regarded as $\lambda_i\colon\mf{a}\to\bb{R}$ in $\Sigma$ by extending it linearly on $\mf{a}_j$ $(j\neq i)$ as $0$. Hence $\mf{g}_{\lambda_i}=(\mf{g}_i)_{\lambda_i}$ for $\lambda_i\in\Sigma_i$. Since any two simple systems of $\Sigma$ are conjugate by $\Ad(k)$ for some $k\in N_K(\mf{a})=N_{K_1}(\mf{a}_1)\times\cdots\times N_{K_\ell}(\mf{a}_\ell)$, it follows that $\Sigma_+=\Sigma_{1+}\cup\cdots\cup\Sigma_{\ell+}$ for some positive system $\Sigma_{i+}$ of $\Sigma_i$. Hence $\mf{n}=\mf{n}_1\oplus\cdots\oplus\mf{n}_\ell$, where $\mf{n}_i=\bigoplus_{\lambda_i\in\Sigma_{i+}}(\mf{g}_i)_{\lambda_i}$. Of course we also have 
\begin{gather*}
A=A_1\times\cdots\times A_\ell,\quad N=N_1\times\cdots\times N_\ell,\quad AN=A_1N_1\times\cdots\times A_\ell N_\ell. 
\end{gather*}
The metric $g$ on $AN$ decomposes as $g=g_1\times\cdots\times g_\ell$, where $g_i$ on $A_iN_i$ is defined by the identification $A_iN_i\simeq G_i/K_i$, $a_in_i\mapsto a_in_iK_i$. The same kind of decomposition holds for the metric on $A$. 

The map $a_\rho(x,\cdotp)\colon G/K\to G/K$ is a quasiisometry. By Kleiner and Leeb \cite[Theorem 1.1.2]{KL} there exist a permutation $\sigma\in\mf{S}_\ell$ and quasiisometries 
\begin{equation*}
\Phi_i\colon\left(G_i/K_i,g_i\right)\to\left(G_{\sigma(i)}/K_{\sigma(i)},g_{\sigma(i)}\right)
\end{equation*}
such that $a_\rho(x,\cdotp)$ and 
\begin{equation*}
\Phi\colon(x_1,\ldots,x_\ell)\mapsto\left(\Phi_{\sigma^{-1}(1)}\left(x_{\sigma^{-1}(1)}\right),\ldots,\Phi_{\sigma^{-1}(\ell)}\left(x_{\sigma^{-1}(\ell)}\right)\right)
\end{equation*}
are close. Then 
\begin{equation*}
\begin{tikzcd}
\Pi A_iN_i\ar[r,"\Phi"]\ar[d,"p"']&\Pi A_iN_i\ar[d,"p"]\\
\Pi A_i\ar[r,"\tilde{\varphi}_\rho"',"\sim"]&\Pi A_i
\end{tikzcd}
\end{equation*}
is fiber respecting. In fact, let $C>0$ be a constant such that 
\begin{equation*}
d_\mca{H}\left(a_\rho(x,aN),\tilde{\varphi}_\rho(a)N\right)<C
\end{equation*}
for all $a\in A$ and $C^\prime>0$ be such that 
\begin{equation*}
d\left(\Phi(s),a_\rho(x,s)\right)<C^\prime
\end{equation*}
for all $s\in AN$, then we have 
\begin{align*}
d_\mca{H}\left(\Phi(aN),\tilde{\varphi}_\rho(a)N\right)&\leq d_\mca{H}\left(\Phi(aN),a_\rho(x,aN)\right)+d_\mca{H}\left(a_\rho(x,aN),\tilde{\varphi}_\rho(a)N\right)\\ 
&<C^\prime+C
\end{align*}
for all $a\in A$. 

\begin{lem}
There exist linear isomorphisms $\varphi_i\colon\mf{a}_i\to\mf{a}_{\sigma(i)}$ such that 
\begin{equation*}
\tilde{\varphi}_\rho(a_1,\ldots,a_\ell)=\left(\tilde{\varphi}_{\sigma^{-1}(1)}\left(a_{\sigma^{-1}(1)}\right),\ldots,\tilde{\varphi}_{\sigma^{-1}(\ell)}\left(a_{\sigma^{-1}(\ell)}\right)\right)
\end{equation*}
for all $a_i\in A_i$ and 
\begin{equation*}
\begin{tikzcd}
A_iN_i\ar[r,"\Phi_i"]\ar[d,"p_i"']&A_{\sigma(i)}N_{\sigma(i)}\ar[d,"p_{\sigma(i)}"]\\
A_i\ar[r,"\tilde{\varphi}_i"',"\sim"]&A_{\sigma(i)}
\end{tikzcd}
\end{equation*}
is fiber respecting, where $p_i$ is the natural projection and $\tilde{\varphi}_i$ is the isomorphism with differential $\varphi_i$. 
\end{lem}

\begin{proof}
Let $\tilde{\varphi}_\rho(a_1,\ldots,a_\ell)=\left(\tilde{\phi}_1(a_1,\ldots,a_\ell),\ldots,\tilde{\phi}_\ell(a_1,\ldots,a_\ell)\right)$ for $a_i\in A_i$. For fixed $i$ and for any $H_i\in\mf{a}_i$ and $t\in\bb{R}$, the Hausdorff distance between 
\begin{align}
\Phi&\left(p^{-1}\left(1,\ldots,1,e^{tH_i},1,\ldots,1\right)\right)\nonumber\\
&=\Phi\left(N_1\times\cdots\times N_{i-1}\times e^{tH_i}N_i\times N_{i+1}\times\cdots\times N_\ell\right)\nonumber\\ 
&=\Phi_{\sigma^{-1}(1)}\left(N_{\sigma^{-1}(1)}\right)\times\cdots\times\Phi_i\left(e^{tH_i}N_i\right)\times\cdots\times\Phi_{\sigma^{-1}(\ell)}\left(N_{\sigma^{-1}(\ell)}\right)\label{nnn}
\end{align}
and 
\begin{align}
\tilde{\phi}_1&\left(1,\ldots,e^{tH_i},\ldots,1\right)N_1\times\cdots\times\tilde{\phi}_\ell\left(1,\ldots,e^{tH_i},\ldots,1\right)N_\ell\nonumber\\ 
&=e^{t\phi_1(0,\ldots,H_i,\ldots,0)}N_1\times\cdots\times e^{t\phi_\ell(0,\ldots,H_i,\ldots,0)}N_\ell\label{mmm}
\end{align}
is bounded by a constant $C>0$, where $\phi_j$ is the differential of $\tilde{\phi}_j$. Thus 
\begin{equation*}
d_\mca{H}\left(\Phi_{\sigma^{-1}(j)}\left(N_{\sigma^{-1}(j)}\right),e^{t\phi_j(0,\ldots,H_i,\ldots,0)}N_j\right)<C
\end{equation*}
for $j\neq\sigma(i)$. Hence 
\begin{equation*}
d\left(e^{t\phi_j(0,\ldots,H_i,\ldots,0)},1\right)=d_\mca{H}\left(e^{t\phi_j(0,\ldots,H_i,\ldots,0)}N_j,N_j\right)<2C
\end{equation*}
for all $t\in\bb{R}$, which implies $\phi_j\left(0,\ldots,H_i,\ldots,0\right)=0$. Therefore, 
\begin{align*}
\tilde{\varphi}_\rho\left(e^{H_1},\ldots,e^{H_\ell}\right)&=\left(\tilde{\phi}_1\left(e^{(H_1,\ldots,H_\ell)}\right),\ldots,\tilde{\phi}_\ell\left(e^{(H_1,\ldots,H_\ell)}\right)\right)\\
&=\left(e^{\phi_1\left(0,\ldots,H_{\sigma^{-1}(1)},\ldots,0\right)},\ldots,e^{\phi_\ell\left(0,\ldots,H_{\sigma^{-1}(\ell)},\ldots,0\right)}\right)\\
&=\left(\tilde{\phi}_1\left(e^{H_{\sigma^{-1}(1)}}\right),\ldots,\tilde{\phi}_\ell\left(e^{H_{\sigma^{-1}(\ell)}}\right)\right)\\ 
&=\left(\tilde{\varphi}_{\sigma^{-1}(1)}\left(e^{H_{\sigma^{-1}(1)}}\right),\ldots,\tilde{\varphi}_{\sigma^{-1}(\ell)}\left(e^{H_{\sigma^{-1}(\ell)}}\right)\right), 
\end{align*}
where we put $\tilde{\varphi}_j=\tilde{\phi}_{\sigma(j)}\colon A_j\to A_{\sigma(j)}$. Finally by looking at $\sigma(i)$-th components of \eqref{nnn} and \eqref{mmm}, we have 
\begin{equation*}
d_\mca{H}\left(\Phi_i\left(e^{tH_i}N_i\right),\tilde{\varphi}_i\left(e^{tH_i}\right)N_{\sigma(i)}\right)=d_\mca{H}\left(\Phi_i\left(e^{tH_i}N_i\right),e^{t\phi_{\sigma(i)}(0,\ldots,H_i,\ldots,0)}N_{\sigma(i)}\right)<C, 
\end{equation*}
hence $\Phi_i$ is fiber respecting over $\tilde{\varphi}_i$. 
\end{proof}

Therefore, Proposition \ref{qq} follows if $\lambda\circ\varphi_{\sigma^{-1}(i)}\in\Sigma_{\sigma^{-1}(i)+}$ for all $\lambda\in\Sigma_{i+}$ since $\Sigma_+=\Sigma_{1+}\cup\cdots\cup\Sigma_{\ell+}$ and $\lambda\circ\varphi_\rho|_\mf{a}=\lambda\circ\varphi_{\sigma^{-1}(i)}$ for $\lambda\in\Sigma_{i+}$. 

Since $G_{\sigma(i)}/K_{\sigma(i)}$ and $G_i/K_i$ are quasiisometric, $\mf{g}_{\sigma(i)}$ and $\mf{g}_i$ are isomorphic. Fix an isomorphism $\alpha\colon\mf{g}_{\sigma(i)}\simeq\mf{g}_i$ such that 
\begin{gather*}
\alpha\left(\mf{k}_{\sigma(i)}\right)=\mf{k}_i,\quad\alpha\left(\mf{p}_{\sigma(i)}\right)=\mf{p}_i,\quad\alpha\left(\mf{a}_{\sigma(i)}\right)=\mf{a}_i
\end{gather*}
and $\alpha$ takes $\Sigma_{\sigma(i)+}$ to $\Sigma_{i+}$. Then $\alpha$ canonically induces isomorphisms 
\begin{gather*}
\mf{n}_{\sigma(i)}\simeq\mf{n}_i, \ G_{\sigma(i)}\simeq G_i, \ K_{\sigma(i)}\simeq K_i, \ A_{\sigma(i)}\simeq A_i, \ N_{\sigma(i)}\simeq N_i
\end{gather*}
and isometries 
\begin{gather*}
\left(G_{\sigma(i)}/K_{\sigma(i)},g_{\sigma(i)}\right)\simeq\left(G_i/K_i,g_i\right),\quad A_{\sigma(i)}N_{\sigma(i)}\simeq A_iN_i. 
\end{gather*}
In this way we identify $A_{\sigma(i)}N_{\sigma(i)}$ with $A_iN_i$ etc. Hence now 
\begin{equation*}
\begin{tikzcd}
A_iN_i\ar[r,"\Phi_i"]\ar[d,"p_i"']&A_iN_i\ar[d,"p_i"]\\
A_i\ar[r,"\tilde{\varphi}_i"',"\sim"]&A_i
\end{tikzcd}
\end{equation*}
is fiber respecting and to complete the proof of Proposition \ref{qq} it suffices to show $\lambda\circ\varphi_i\in\Sigma_{i+}$ for any $\lambda\in\Sigma_{i+}$. 

We consider the following two cases separately: 
\begin{itemize}
\item The group $G_i$ is of real rank at least $2$ or locally isomorphic to $\Sp(n,1)$ $(n\geq2)$ or $F_4^{-20}$. 
\item The group $G_i$ is of real rank $1$. 
\end{itemize}
We can treat $G_i$ locally isomorphic to $\Sp(n,1)$ $(n\geq2)$ or $F_4^{-20}$ in either cases. 

From now on we will no longer consider the original objects $G$, $K$, $A$, $N$, $g$, $\Sigma$, $\varphi_\rho$ etc and we will focus only on the decomposed objects $G_i$, $K_i$, $A_i$, $N_i$, $g_i$, $\Sigma_i$, $\varphi_i$ etc. Hence we will drop all the subscripts $i$ to simplify the notations. So we have 
\begin{equation*}
G, \mf{g}, K, \mf{k}, \mf{p}, \theta, B, A, \mf{a}, N, \mf{n}, \mf{n}^\prime, \Sigma, \Sigma_+, \mf{g}_\lambda, \mf{g}_\lambda^\prime, g, \varphi, \tilde{\varphi}, \Phi, p, 
\end{equation*}
but we do not have $M$, $\rho$ and $a_\rho$. Recall that $G=\Ad(G)$, $g$ is the restriction of $B$ at $T_KG/K=\mf{p}$, $AN$ is equipped with a Riemannian metric by the identification $AN\simeq G/K$, the Riemannian metric of $A$ is the one which makes $p_*|_\mf{a}\colon\mf{a}\simeq\mf{a}$ an isometry, and 
\begin{equation*}
\begin{tikzcd}
AN\ar[r,"\Phi"]\ar[d,"p"']&AN\ar[d,"p"]\\
A\ar[r,"\tilde{\varphi}"',"\sim"]&A
\end{tikzcd}
\end{equation*}
is fiber respecting. Under these conditions we must prove $\lambda\circ\varphi\in\Sigma_+$ for any $\lambda\in\Sigma_+$.

\subsection{The case where $G$ is of real rank at least $2$ or locally isomorphic to $\Sp(n,1)$ $(n\geq2)$ or $F_4^{-20}$}
By Kleiner--Leeb \cite[Theorem 1.1.3]{KL} for $G$ of real rank at least $2$ and by Pansu \cite[1. Th\'{e}or\`{e}me]{Pa} for $G$ locally isomorphic to $\Sp(n,1)$ $(n\geq2)$ or $F_4^{-20}$, there exists a homothety 
\begin{equation*}
F\colon(G/K,g)\to(G/K,g)
\end{equation*}
close to $\Phi$. Thus there is a constant $c>0$ such that $g\left(F_*X,F_*Y\right)=cg(X,Y)$ for all $x\in G/K$ and $X$, $Y\in T_xG/K$, so $F\colon(G/K,cg)\to(G/K,g)$ is an isometry. Since the isometry group of $G/K$ acts transitively, there exists the minimum $K_0\in(-\infty,0)$ of the sectional curvature of $(G/K,g)$. Then $cK_0$ is the minimum of the sectional curvature of $(G/K,cg)$. Since they are isometric we must have $K_0=cK_0$ hence $c=1$. Thus $F\colon(G/K,g)\to(G/K,g)$ is an isometry. 

Since $F$ is close to $\Phi$, 
\begin{equation*}
\begin{tikzcd}
AN\ar[r,"F"]\ar[d,"p"']&AN\ar[d,"p"]\\
A\ar[r,"\tilde{\varphi}"']&A
\end{tikzcd}
\end{equation*}
is fiber respecting. 

Let $F(1)^{-1}=a_0n_0\in AN$. We have 
\begin{equation*}
\begin{tikzcd}
AN\ar[r,"L_{a_0n_0}"]\ar[d,"p"']\ar[dr,phantom,"\circlearrowright" marking]&AN\ar[d,"p"]\\
A\ar[r,"L_{a_0}"']&A, 
\end{tikzcd}
\end{equation*}
where $L$ denotes the left multiplication. Since $L_{a_0n_0}$ is an isometry, 
\begin{equation*}
\begin{tikzcd}
AN\ar[r,"f"]\ar[d,"p"']&AN\ar[d,"p"]\\
A\ar[r,"L_{a_0}\circ\tilde{\varphi}"']&A
\end{tikzcd}
\end{equation*}
is fiber respecting, where $f=L_{a_0n_0}\circ F$. Since $L_{a_0}\circ\tilde{\varphi}$ and $\tilde{\varphi}$ are close, 
\begin{equation*}
\begin{tikzcd}
AN\ar[r,"f"]\ar[d,"p"']&AN\ar[d,"p"]\\
A\ar[r,"\tilde{\varphi}"']&A
\end{tikzcd}
\end{equation*}
is also fiber respecting. Note that $f$ is an isometry and $f(1)=1$. 

\begin{lem}
The map $\tilde{\varphi}$ is an isometry. 
\end{lem}

\begin{proof}
There exists a constant $C>0$ such that $d_\mca{H}\left(f(aN),\tilde{\varphi}(a)N\right)<C$ for all $a\in A$. Then we have 
\begin{align*}
\left\lvert d(1,a)-d\left(1,\tilde{\varphi}(a)\right)\right\rvert&=\left\lvert d_\mca{H}\left(f(N),f(aN)\right)-d_\mca{H}\left(N,\tilde{\varphi}(a)N\right)\right\rvert \\
&\leq\left\lvert d_\mca{H}\left(f(N),f(aN)\right)-d_\mca{H}\left(f(N),\tilde{\varphi}(a)N\right)\right\rvert \\
&\ \ \ \ +\left\lvert d_\mca{H}\left(f(N),\tilde{\varphi}(a)N\right)-d_\mca{H}\left(N,\tilde{\varphi}(a)N\right)\right\rvert \\
&\leq d_\mca{H}\left(f(aN),\tilde{\varphi}(a)N\right)+d_\mca{H}\left(f(N),N\right)\\
&<2C
\end{align*}
for all $a\in A$. Hence for all $t>0$ and $H\in\mf{a}$ we have 
\begin{equation*}
\left\lvert d\left(1,e^{tH}\right)-d\left(1,e^{t\varphi H}\right)\right\rvert<2C, 
\end{equation*}
that is, 
\begin{equation*}
\left\lvert t\lVert H\rVert-t\lVert\varphi H\rVert\right\rvert<2C. 
\end{equation*}
This implies 
\begin{equation*}
\lVert\varphi H\rVert=\lVert H\rVert. 
\end{equation*}
Thus $\tilde{\varphi}$ is an isometry. 
\end{proof}

Now we regard $f$ as $f\colon G/K\to G/K$ and $p\colon G/K\to A$. Consider 
\begin{equation*}
f_*\colon\mf{p}=T_KG/K\to\mf{p}=T_KG/K. 
\end{equation*}

\begin{lem}\label{jj}
We have $f_*(\mf{a})=\mf{a}$ and $f_*|_\mf{a}=\varphi\colon\mf{a}\to\mf{a}$. 
\end{lem}

\begin{proof}
Take any $H\in\mf{a}$. Let $f_*H=X+Y$ for some $X\in\mf{a}$ and $Y\in\mf{n}^\prime$. Since 
\begin{equation*}
\lVert H\rVert^2=\lVert f_*H\rVert^2=\lVert X\rVert^2+\lVert Y\rVert^2\geq\lVert X\rVert^2, 
\end{equation*}
we have $\lVert H\rVert\geq\lVert X\rVert$. Because $e^{tH}K$ $(t\in\bb{R})$ is a geodesic and $f$ is an isometry, $f\left(e^{tH}K\right)$ $(t\in\bb{R})$ is also a geodesic and $f\left(e^{tH}K\right)=e^{tX+tY}K$. Let $t>0$. 

Since $f$ is fiber respecting over $\tilde{\varphi}$, there exists a constant $C>0$ such that 
\begin{equation*}
d_\mca{H}\left(f\left(p^{-1}\left(e^{tH}\right)\right),p^{-1}\left(\tilde{\varphi}\left(e^{tH}\right)\right)\right)<C. 
\end{equation*}

Since $e^{tX+tY}K=f\left(e^{tH}K\right)\in f\left(p^{-1}\left(e^{tH}\right)\right)$ and by the definition of the Hausdorff distance, there exists $x\in p^{-1}\left(\tilde{\varphi}\left(e^{tH}\right)\right)$ such that $d\left(e^{tX+tY}K,x\right)<C$. 

The map $p$ is distance decreasing since $d(a,a^\prime)=d\left(p^{-1}(a),p^{-1}(a^\prime)\right)$ for all $a$, $a^\prime\in A$. So 
\begin{equation}\label{uouo}
d\left(e^{tX},\tilde{\varphi}\left(e^{tH}\right)\right)=d\left(p\left(e^{tX+tY}K\right),p(x)\right)\leq d\left(e^{tX+tY}K,x\right)<C. 
\end{equation}
Since $\tilde{\varphi}$ is an isometry, 
\begin{equation*}
d\left(\tilde{\varphi}(e^{tH}),1\right)=d\left(e^{tH},1\right)=t\lVert H\rVert. 
\end{equation*}
By the triangle inequality we have 
\begin{equation*}
t\lVert H\rVert-t\lVert X\rVert=d\left(1,\tilde{\varphi}\left(e^{tH}\right)\right)-d\left(1,e^{tX}\right)\leq d\left(e^{tX},\tilde{\varphi}\left(e^{tH}\right)\right)<C
\end{equation*}
for all $t>0$. This forces $\lVert H\rVert=\lVert X\rVert$ and then $Y=0$ by the equation $\lVert H\rVert^2=\lVert X\rVert^2+\lVert Y\rVert^2$. Hence $f_*H=X\in\mf{a}$. 

For the second assertion we have by \eqref{uouo} 
\begin{equation*}
d\left(e^{tf_*H},e^{t\varphi H}\right)<C
\end{equation*}
for any $t\in\bb{R}$. This implies $f_*H=\varphi H$. 
\end{proof}

\begin{prop}\label{helga}
Let $\mf{g}$ be a real semisimple Lie algebra and let $G=\Inn(\mf{g})$. (Recall that the Lie algebra of $G$ is naturally isomorphic to $\mf{g}$ and $G$ is the identity component of $\Aut(\mf{g})$.) Fix a maximal compact subgroup $K$ of $G$: 
\begin{enumerate}
\item Let $\psi\in\Aut(\mf{g})$ and consider $\Psi\in\Aut(G)$ defined by $\Psi(g)=\psi g\psi^{-1}$. The automorphism $\Psi$ permutes the maximal compact subgroups of $G$. Identifying the set of all maximal compact subgroups of $G$ with $G/K$ by $gKg^{-1}\leftrightarrow gK$, the map $I_\psi:G/K\to G/K$ induced by $\Psi$ is an isometry with respect to the $G$-invariant Riemannian metric defined by the restriction of the Killing form to the orthogonal complement of the Lie algebra of $K$. 
\item Suppose $\mf{g}$ has no compact simple factor. Then the mapping $\psi\mapsto I_\psi$ is an isomorphism from $\Aut(\mf{g})$ to $\Isom(G/K)$. 
\end{enumerate}
\end{prop}

\begin{proof}
This is Exercise 7 in Chapter VI of Helgason \cite{Hel}. A proof can be found in Solutions to Exercises. 
\end{proof}

By Proposition \ref{helga} there exists $\psi\in\Aut(\mf{g})$ such that $f=I_\psi$. Since $f(K)=K$, we have $\Psi(K)=K$. This implies 
\begin{equation}\label{form}
f\left(gK\right)=\Psi(g)K
\end{equation}
for all $g\in G$. We have $\psi(\mf{k})=\mf{k}$. Since 
\begin{equation*}
\mf{p}=\left\lbrace X\in\mf{g}\mid B(X,Y)=0 \ \text{for all} \ Y\in\mf{k}\right\rbrace
\end{equation*}
and $B$ is $\psi$-invariant, we also have $\psi(\mf{p})=\mf{p}$. Hence $f_*=\psi|_\mf{p}\colon\mf{p}\to\mf{p}$ by \eqref{form} and $\psi(\mf{a})=\mf{a}$ by Lemma \ref{jj}. Therefore, $\psi|_\mf{a}=\varphi\colon\mf{a}\to\mf{a}$ again by Lemma \ref{jj}. Since $\psi$ is an isomorphism of $\mf{g}$ which preserves $\mf{a}$, we have $\psi^{-1}\mf{g}_\lambda=\mf{g}_{\lambda\circ\psi|_\mf{a}}$ for any $\lambda\in\Sigma$. Thus $\lambda\circ\varphi=\lambda\circ\psi|_\mf{a}\in\Sigma$ if $\lambda\in\Sigma$. We must show that $\lambda\circ\varphi=\lambda\circ\psi|_\mf{a}\in\Sigma_+$ if $\lambda\in\Sigma_+$. 

For a Weyl chamber $C$ in $\mf{a}$, let 
\begin{equation*}
\Sigma_C=\lbrace\lambda\in\Sigma\mid\text{$\lambda(H)>0$ for some $H\in C$}\rbrace
\end{equation*}
be the positive system corresponding to $C$, let $\mf{n}_C=\bigoplus_{\lambda\in\Sigma_C}\mf{g}_\lambda$, and let $N_C$ be the Lie subgroup corresponding to $\mf{n}_C$. 

Let $C_0\subset\mf{a}$ be the Weyl chamber corresponding to $\Sigma_+$, ie 
\begin{equation*}
C_0=\left\lbrace H\in\mf{a}\mid\text{$\lambda(H)>0$ for all $\lambda\in\Sigma_+$}\right\rbrace. 
\end{equation*}
Then $C_1=\psi C_0$ is a Weyl chamber in $\mf{a}$. We have $\lambda\in\Sigma_+$ if and only if $\lambda\circ(\psi|_\mf{a})^{-1}\in\Sigma_{C_1}$. Thus $\psi\mf{n}=\mf{n}_{C_1}$. By \eqref{form} we have $f(NK)=N_{C_1}K$. Therefore, the Hausdorff distance between $N_{C_1}K$ and $NK$ is finite.  

\begin{lem}
If $C$ and $C^\prime$ are distinct Weyl chambers in $\mf{a}$, then the Hausdorff distance between $N_CK$ and $N_{C^\prime}K$ is infinite. 
\end{lem}

\begin{proof}
Take $\lambda\in\Sigma_C\setminus\Sigma_{C^\prime}$. Hence $\mf{g}_\lambda\subset\mf{n}_C$ and $\mf{g}_{-\lambda}\subset\mf{n}_{C^\prime}$. We will prove that $e^{\mf{g}_{-\lambda}}K$ contains arbitrarily far points from $N_CK$. Let $H_\lambda\in\mf{a}$ be the element defined by $\lambda(H)=B\left(H_\lambda,H\right)$ for all $H\in\mf{a}$. By Knapp \cite[Proposition 6.52]{Kn} there exists nonzero $X_\lambda\in\mf{g}_\lambda$ such that: 
\begin{itemize}
\item $\left[X_\lambda,\theta X_\lambda\right]=B\left(X_\lambda,\theta X_\lambda\right)H_\lambda$
\item $B\left(X_\lambda,\theta X_\lambda\right)=-\frac{2}{B(H_\lambda,H_\lambda)}<0$
\item the subspace $\bb{R}\theta X_\lambda\oplus\bb{R}H_\lambda\oplus\bb{R}X_\lambda$ is a Lie subalgebra of $\mf{g}$ isomorphic to $\mf{sl}(2,\bb{R})$. The isomorphism is given by 
\begin{align*}
X_{-\lambda}^\prime=\theta X_\lambda\quad&\longleftrightarrow\quad
\begin{pmatrix}
0&0\\
1&0
\end{pmatrix}\\
H_{\lambda}^\prime=\frac{2}{B(H_\lambda,H_\lambda)}H_\lambda\quad&\longleftrightarrow\quad
\begin{pmatrix}
1&0\\
0&-1
\end{pmatrix}\\
X_{\lambda}^\prime=-X_\lambda\quad&\longleftrightarrow\quad
\begin{pmatrix}
0&1\\
0&0
\end{pmatrix}
. 
\end{align*}
\end{itemize}

For any $x\in\bb{R}$ we have 
\begin{equation*}
\begin{pmatrix}
1&0\\
x&1
\end{pmatrix}
=
\begin{pmatrix}
1&\frac{x}{1+x^2}\\
0&1
\end{pmatrix}
\begin{pmatrix}
\frac{1}{\sqrt{1+x^2}}&0\\
0&\sqrt{1+x^2}
\end{pmatrix}
\begin{pmatrix}
\frac{1}{\sqrt{1+x^2}}&\frac{x}{\sqrt{1+x^2}}\\
-\frac{x}{\sqrt{1+x^2}}&\frac{1}{\sqrt{1+x^2}}
\end{pmatrix}. 
\end{equation*}
This can be regarded as an equation of elements in the universal cover $\widetilde{\SL}(2,\bb{R})$ of $\SL(2,\bb{R})$. We rewrite it using the exponential map: 
\begin{align*}
\exp
\begin{pmatrix}
0&0\\
x&0
\end{pmatrix}
&=\exp
\begin{pmatrix}
0&\frac{x}{1+x^2}\\
0&0
\end{pmatrix}
\exp
\begin{pmatrix}
-\frac{\log(1+x^2)}{2}&0\\
0&\frac{\log(1+x^2)}{2}
\end{pmatrix}
\\
&\qquad\qquad\cdot\exp
\begin{pmatrix}
0&-\arctan x\\
\arctan x&0
\end{pmatrix}. 
\end{align*}
Mapping the above equation by the homomorphism $\widetilde{\SL}(2,\bb{R})\to G$, we get 
\begin{align}\label{iwa}
\exp\left(xX_{-\lambda}^\prime\right)&=\exp\left(\frac{x}{1+x^2}X_\lambda^\prime\right)\exp\left(-\frac{\log(1+x^2)}{2}H_\lambda^\prime\right)\nonumber\\
&\qquad\qquad\cdot\exp\left(\arctan x\left(X_{-\lambda}^\prime-X_\lambda^\prime\right)\right). 
\end{align}
Note that 
\begin{gather*}
\exp\left(xX_{-\lambda}^\prime\right)\in e^{\mf{g}_{-\lambda}}\subset N_{C^\prime},\quad\exp\left(\frac{x}{1+x^2}X_\lambda^\prime\right)\in e^{\mf{g}_\lambda}\subset N_C,\\
\exp\left(-\frac{\log(1+x^2)}{2}H_\lambda^\prime\right)\in A. 
\end{gather*}
Since $\theta\left(X_{-\lambda}^\prime-X_\lambda^\prime\right)=X_{-\lambda}^\prime-X_\lambda^\prime$, we have $X_{-\lambda}^\prime-X_\lambda^\prime\in\mf{k}$, hence 
\begin{equation*}
\exp\left(\arctan x\left(X_{-\lambda}^\prime-X_\lambda^\prime\right)\right)\in K. 
\end{equation*}
Thus \eqref{iwa} gives the Iwasawa decomposition of $\exp\left(xX_{-\lambda}^\prime\right)$ as an element of $G=N_CAK$. Therefore, 
\begin{align*}
d\left(\exp\left(xX_{-\lambda}^\prime\right)K,N_CK\right)&=d\left(\exp\left(\frac{x}{1+x^2}X_\lambda^\prime\right)\exp\left(-\frac{\log(1+x^2)}{2}H_\lambda^\prime\right)K,N_CK\right)\\
&=d\left(\exp\left(-\frac{\log(1+x^2)}{2}H_\lambda^\prime\right)K,N_CK\right)\\
&=d\left(\exp\left(-\frac{\log(1+x^2)}{2}H_\lambda^\prime\right)K,K\right)\\
&=\left\lVert-\frac{\log(1+x^2)}{2}H_\lambda^\prime\right\rVert\\
&=\frac{\log(1+x^2)}{\sqrt{B(H_\lambda,H_\lambda)}}. 
\end{align*}
This shows $N_{C^\prime}K$ contains points arbitrarily far from $N_CK$. 
\end{proof}

Thus $C_1=C_0$ and so $\psi\mf{n}=\mf{n}$. Hence $\lambda\circ\varphi=\lambda\circ\psi|_\mf{a}\in\Sigma_+$ if $\lambda\in\Sigma_+$.

\subsection{The case where $G$ is of real rank $1$}
\begin{prop}\label{hyperbolic}
If 
\begin{equation*}
\begin{tikzcd}
AN\ar[r,"f"]\ar[d,"p"']&AN\ar[d,"p"]\\
A\ar[r,"h"']&A
\end{tikzcd}
\end{equation*}
is fiber respecting, $f$ is a quasiisometry and $h$ is a map, then $h$ is close to the identity map. 
\end{prop}

The map $\tilde{\varphi}$ is close to the identity map by this proposition. But since $\tilde{\varphi}$ is a homomorphism, $\tilde{\varphi}$ must be the identity map. Hence $\lambda\circ\varphi=\lambda\in\Sigma_+$ for all $\lambda\in\Sigma_+$ and this concludes the proof of Proposition \ref{qq}. 

Proposition \ref{hyperbolic} is Proposition 5.8 of Farb--Mosher \cite{FM} when $G$ is locally isomorphic to $\SO_0(n,1)$. For the other cases it is basically Theorem 33 of Reiter Ahlin \cite{RA} but the proof there seems incomplete. To get the conclusion of Proposition \ref{hyperbolic} we need to argue at some point in the same manner as Farb--Mosher do. Here we give a proof of Proposition \ref{hyperbolic} following the arguments by Farb--Mosher and Reiter Ahlin. 

We have $\Sigma_+=\lbrace\lambda\rbrace$ for $G$ locally isomorphic to $\SO_0(n,1)$ and $\Sigma_+=\left\lbrace\lambda,2\lambda\right\rbrace$ for the other cases. Accordingly $\mf{n}=\mf{g}_\lambda$ in the former case and $\mf{n}=\mf{g}_\lambda\oplus\mf{g}_{2\lambda}$ in the latter case. Take $H\in\mf{a}$ such that $\lambda(H)=1$. Hence $\mf{a}=\bb{R}H$. We identify $A$ with $\bb{R}$ by $e^{tH}\to t$. 

We write the proof for the case of $\Sigma_+=\lbrace\lambda,2\lambda\rbrace$ but no change is needed when we have $\Sigma_+=\lbrace\lambda\rbrace$ except notational one. 

Let $g_t$ be the Riemannian metric on $N$ induced from $g$ by the embedding $N\hookrightarrow AN$, $x\mapsto xe^{tH}$. Let $d$ and $d_t$ be the metrics induced from $g$ and $g_t$ respectively. Since $x\left(ye^{tH}\right)=\left(xy\right)e^{tH}$, ie the embedding $N\hookrightarrow AN$ is $N$-equivariant, $g_t$ is a left invariant Riemannian metric on $N$. Let $\lVert\cdotp\rVert_j$ be a norm on $\mf{g}_{j\lambda}$ $(j=1,2)$ and set $\lvert x\rvert=\max\left\lbrace\left\lVert \xi\right\rVert_1,\left\lVert v\right\rVert_2^\frac{1}{2}\right\rbrace$ for $x\in N$, where $\log x=\xi+v$ for $\xi\in\mf{g}_\lambda$, $v\in\mf{g}_{2\lambda}$. Let $\phi_t\colon N\to N$ be the map defined by $\phi_t(x)=e^{tH}xe^{-tH}$. Then 
\begin{align*}
\left\lvert\phi_t(x)\right\rvert&=\left\lvert e^{tH}e^{\xi+v}e^{-tH}\right\rvert=\left\lvert\exp\left(e^{t\ad H}(\xi+v)\right)\right\rvert\\
&=\left\lvert\exp\left(e^t\xi+e^{2t}v\right)\right\rvert=\max\left\lbrace e^t\left\lVert\xi\right\rVert_1,e^t\left\lVert v\right\rVert_2^\frac{1}{2}\right\rbrace\\
&=e^t\max\left\lbrace\left\lVert\xi\right\rVert_1,\left\lVert v\right\rVert_2^\frac{1}{2}\right\rbrace=e^t\left\lvert e^{\xi+v}\right\rvert\\
&=e^t\left\lvert x\right\rvert
\end{align*}
for any $x\in N$ and $t\in\bb{R}$. 

\begin{lem}\label{inequ}
There exists $K_1\geq1$ such that 
\begin{equation*}
\frac{1}{K_1}e^{-t}\left\lvert x^{-1}y\right\rvert-K_1\leq d_t(x,y)\leq K_1e^{-t}\left\lvert x^{-1}y\right\rvert+K_1
\end{equation*}
for all $t\in\bb{R}$ and $x$, $y\in N$. 
\end{lem}

\begin{proof}
Since $e^{tH}x=\phi_t(x)e^{tH}$, $\phi_t\colon(N,g_0)\to(N,g_t)$ is an isometry. Hence 
\begin{equation*}
d_t(x,y)=d_t\left(1,x^{-1}y\right)=d_0\left(1,\phi_{-t}\left(x^{-1}y\right)\right). 
\end{equation*}
It is known that there exists a constant $K_1\geq1$ such that 
\begin{equation*}
\frac{1}{K_1}\left\vert x\right\rvert-K_1\leq d_0(1,x)\leq K_1\left\lvert x\right\rvert+K_1
\end{equation*}
for all $x\in N$. See for example Breuillard \cite[Proposition 4.5]{B}. Therefore, 
\begin{equation*}
d_t(x,y)\leq K_1\left\lvert\phi_{-t}\left(x^{-1}y\right)\right\rvert+K_1=K_1e^{-t}\left\lvert x^{-1}y\right\rvert+K_1
\end{equation*}
and
\begin{equation*}
\frac{1}{K_1}e^{-t}\left\lvert x^{-1}y\right\rvert-K_1=\frac{1}{K_1}\left\lvert\phi_{-t}\left(x^{-1}y\right)\right\rvert-K_1\leq d_t(x,y). 
\end{equation*}
\end{proof}

\begin{cor}\label{mimi}
There exists $K_2\geq1$ such that for any fixed $t_0\in\bb{R}$ we have 
\begin{equation*}
\frac{1}{K_2^2}e^{t_0-t}\leq\frac{d_t(x,y)}{d_{t_0}(x,y)}\leq K_2^2e^{t_0-t}
\end{equation*}
if $t\leq t_0$ and $\left\lvert x^{-1}y\right\rvert>\left(K_1^2+1\right)e^{t_0}$. 
\end{cor}

\begin{proof}
If $t\leq t_0$ and $\left\lvert x^{-1}y\right\rvert>\left(K_1^2+1\right)e^{t_0}$, then we have $e^{-t}\left\lvert x^{-1}y\right\rvert>K_1^2+1$, hence 
\begin{equation*}
\left(\frac{1}{K_1}-\frac{K_1}{K_1^2+1}\right)e^{-t}\left\lvert x^{-1}y\right\rvert\leq d_t(x,y)\leq\left(K_1+\frac{K_1}{K_1^2+1}\right)e^{-t}\left\lvert x^{-1}y\right\rvert
\end{equation*}
by Lemma \ref{inequ}. Since 
\begin{equation*}
\frac{1}{K_1}-\frac{K_1}{K_1^2+1}>0, 
\end{equation*}
there exists $K_2\geq1$, which is independent of $t_0$, such that 
\begin{equation*}
\frac{1}{K_2}e^{-t}\left\lvert x^{-1}y\right\rvert\leq d_t(x,y)\leq K_2e^{-t}\left\lvert x^{-1}y\right\rvert
\end{equation*}
under the above conditions. In particular 
\begin{equation*}
\frac{1}{K_2}e^{-t_0}\left\lvert x^{-1}y\right\rvert\leq d_{t_0}(x,y)\leq K_2e^{-t_0}\left\lvert x^{-1}y\right\rvert. 
\end{equation*}
We get the conclusion from these two inequalities. 
\end{proof}

A map $\sigma\colon S\to X$ between geodesic spaces is called {\em uniformly proper} if there exist constants $K\geq1$, $C\geq0$ and a function $\rho\colon\bb{R}_{\geq0}\to\bb{R}_{\geq0}$ with $\lim_{a\to\infty}\rho(a)=\infty$ such that 
\begin{equation*}
\rho\left(d(x,y)\right)\leq d\left(\sigma(x),\sigma(y)\right)\leq Kd(x,y)+C
\end{equation*}
for all $x$, $y\in S$. We call $\rho$, $K$ and $C$ {\em the uniformity data for $\sigma$}. 

\begin{lem}\label{up}
The embedding $\left(N,d_t\right)\hookrightarrow\left(AN,d\right)$ is uniformly proper for each $t\in\bb{R}$ and the uniformity data are independent of $t$. In fact there exists a function $\rho\colon\bb{R}_{\geq0}\to\bb{R}_{\geq0}$ with $\lim_{a\to\infty}\rho(a)=\infty$ such that 
\begin{equation*}
\rho\left(d_t(x,y)\right)\leq d\left(xe^{tH},ye^{tH}\right)\leq d_t(x,y)
\end{equation*}
for all $x$, $y\in N$ and $t\in\bb{R}$. 
\end{lem}

\begin{proof}
The second inequality is obvious. For the first inequality, define $\rho_1\colon\bb{R}_{\geq0}\to\bb{R}_{\geq0}$ by 
\begin{equation*}
\rho_1(R)=\sup\left\lbrace d_0(1,x)\mid x\in N,\ d(1,x)=R\right\rbrace. 
\end{equation*}
Then $\rho_1$ is strictly increasing and $\lim_{R\to\infty}\rho_1(R)=\infty$. We have $d_0(1,x)\leq\rho_1\left(d(1,x)\right)$ for any $x\in N$ hence $d_0(x,y)\leq\rho_1\left(d(x,y)\right)$ for all $x$, $y\in N$. Since 
\begin{align*}
d_t(x,y)&=d_0\left(e^{-tH}xe^{tH},e^{-tH}ye^{tH}\right)\\
&\leq\rho_1\left(d\left(e^{-tH}xe^{tH},e^{-tH}ye^{tH}\right)\right)\\
&=\rho_1\left(d\left(xe^{tH},ye^{tH}\right)\right), 
\end{align*}
we get $\rho_1^{-1}\left(d_t(x,y)\right)\leq d\left(xe^{tH},ye^{tH}\right)$. So $\rho=\rho_1^{-1}$ satisfies the required properties. 
\end{proof}

\begin{lem}\label{fmlem}
Let $X$, $Y$, $S$, $T$ be geodesic spaces, let $f\colon X\to Y$ be a quasiisometry, and let $\sigma\colon S\to X$, $\tau\colon T\to Y$ be uniformly proper maps such that $d_\mca{H}\left(f\sigma(S),\tau(T)\right)<\infty$. Take any map $g\colon S\to T$ satisfying $\sup_{x\in S}d\left(f\sigma(x),\tau g(x)\right)<\infty$. Then $g$ is a quasiisometry and the quasiisometry constants depend only on the quasiisometry constants for $f$, the uniformity data for $\sigma$ and $\tau$, and $\sup_{x\in S}d\left(f\sigma(x),\tau g(x)\right)$. 
\end{lem}

\begin{proof}
This is Lemma 2.1 of Farb and Mosher \cite{FM}. 
\end{proof}

We identify $h\colon A\to A$ with $h\colon\bb{R}\to\bb{R}$ by $h\left(e^{tH}\right)=e^{h(t)H}$. Define $f_t\colon\left(N,d_t\right)\to\left(N,d_{h(t)}\right)$ by $f\left(xe^{tH}\right)=f_t(x)e^{u(x,t)H}$. Then $f_t$ satisfies the property of Lemma \ref{fmlem}. In fact since $f$ is fiber respecting over $h$, there exists a constant $C_1>0$ such that $d_\mca{H}\left(f\left(p^{-1}\left(e^{tH}\right)\right),p^{-1}\left(e^{h(t)H}\right)\right)<C_1$ for all $t\in\bb{R}$. Hence there exists $y\in N$ such that $d\left(f\left(xe^{tH}\right),ye^{h(t)H}\right)<C_1$. Therefore, 
\begin{align}
d\left(f\left(xe^{tH}\right),f_t(x)e^{h(t)H}\right)&=d\left(f_t(x)e^{u(x,t)H},f_t(x)e^{h(t)H}\right)\nonumber\\
&\leq d\left(f_t(x)e^{u(x,t)H},ye^{h(t)H}\right)<C_1\label{ooo}
\end{align}
for all $x\in N$ and $t\in\bb{R}$. By Lemma \ref{up} and Lemma \ref{fmlem}, $f_t\colon\left(N,d_t\right)\to\left(N,d_{h(t)}\right)$ is a quasiisometry with quasiisometry constants independent of $t$. 

Let $\partial AN$ be the Gromov boundary of $AN$. Then $\partial AN=\lbrace\infty\rbrace\cup N$. The quasiisometry $f$ induces a map $\partial f\colon\partial AN\to\partial AN$. 

\begin{lem}\label{fix}
$\partial f(\infty)=\infty$. 
\end{lem}

\begin{proof}
Assume the contrary: $\partial f(\infty)=x\in N$. Take $y\in N$ with $y\neq(\partial f)^{-1}(\infty)$. Let $\gamma$ be the directed geodesic connecting $y$ and $\infty$. Then the Hausdorff distance between $f(\gamma)$ and the directed geodesic $\gamma^\prime$ connecting $\partial f(y)$ and $x$ is finite. Hence the height of $f(\gamma)$ is bounded above. Since $h$ is a quasiisometry, we can choose $t_0\in\bb{R}$ so that $h(t_0)$ is as large as we wish. Therefore, the height of $f\left(p^{-1}\left(e^{t_0H}\right)\right)$ is also large. But we always have $f\left(ye^{t_0H}\right)\in f\left(p^{-1}\left(e^{t_0H}\right)\right)\cap f(\gamma)\neq\varnothing$, which is impossible. 
\end{proof}

For any $x\in N$, $xe^{tH}$ $(t\in\bb{R})$ is a geodesic of $AN$ connecting $x\in\partial AN$ and $\infty$. Then $f\left(xe^{tH}\right)$ $(t\in\bb{R})$ is a quasigeodesic of $AN$. By Lemma \ref{fix} there exists a constant $C_2>0$ such that $d_\mca{H}\left(f\left(xe^{\bb{R}H}\right),\partial f(x)e^{\bb{R}H}\right)<C_2$. By \eqref{ooo} 
\begin{equation}
\left\lvert u(x,t)-h(t)\right\rvert\lVert H\rVert<C_1. \label{kaka}
\end{equation}
There exists $s(x,t)\in\bb{R}$ such that $d\left(f\left(xe^{tH}\right),\partial f(x)e^{s(x,t)H}\right)<C_2$. We have 
\begin{align}
\left\lvert u(x,t)-s(x,t)\right\rvert\lVert H\rVert&=d\left(\partial f(x)e^{u(x,t)H},\partial f(x)e^{s(x,t)H}\right)\nonumber\\
&=d\left(p^{-1}\left(e^{u(x,t)H}\right),\partial f(x)e^{s(x,t)H}\right)\nonumber\\
&\leq d\left(f\left(xe^{tH}\right),\partial f(x)e^{s(x,t)H}\right)<C_2. \label{lolo}
\end{align}
By \eqref{kaka} and \eqref{lolo} we get 
\begin{equation}\label{chikai}
\left\lvert s(x,t)-h(t)\right\rvert\lVert H\rVert<C_1+C_2. 
\end{equation}
Therefore, 
\begin{align*}
d&\left(f_t(x)e^{h(t)H},\partial f(x)e^{h(t)H}\right)\\
&\leq d\left(f_t(x)e^{h(t)H},f_t(x)e^{u(x,t)H}\right)+d\left(f_t(x)e^{u(x,t)H},\partial f(x)e^{s(x,t)H}\right)\\
&\qquad\qquad+d\left(\partial f(x)e^{s(x,t)H},\partial f(x)e^{h(t)H}\right)\\
&<\left\lvert u(x,t)-h(t)\right\rvert\lVert H\rVert+C_2+\left\lvert s(x,t)-h(t)\right\rvert\lVert H\rVert\\
&<2C_1+2C_2. 
\end{align*}
Hence 
\begin{equation*}
d_{h(t)}\left(f_t(x),\partial f(x)\right)\leq\rho^{-1}\left(d\left(f_t(x)e^{h(t)H},\partial f(x)e^{h(t)H}\right)\right)<\rho^{-1}\left(2C_1+2C_2\right). 
\end{equation*}
Namely $f_t$ and $\partial f$ are close and the constant of closeness is independent of $t$. Thus $\partial f\colon\left(N,d_t\right)\to\left(N,d_{h(t)}\right)$ is a quasiisometry with constants independent of $t$, so there exists a constant $K_3\geq1$ such that 
\begin{equation*}
\frac{1}{K_3}d_t(x,y)-K_3\leq d_{h(t)}\left(\partial f(x),\partial f(y)\right)\leq K_3d_t(x,y)+K_3
\end{equation*}
for all $x$, $y\in N$ and $t\in\bb{R}$. 

\begin{lem}\label{lulu}
For any fixed $t_0\in\bb{R}$ we have 
\begin{equation*}
\frac{1}{2K_3}d_t(x,y)\leq d_{h(t)}\left(\partial f(x),\partial f(y)\right)\leq2K_3d_t(x,y)
\end{equation*}
for all $t\leq t_0$ and $x$, $y\in N$ with $\left\lvert x^{-1}y\right\rvert>e^{t_0}K_1\left(2K_3^2+K_1\right)$. 
\end{lem}

\begin{proof}
If $t\leq t_0$ and $\left\lvert x^{-1}y\right\rvert>e^{t_0}K_1\left(2K_3^2+K_1\right)$, we have 
\begin{equation*}
d_t(x,y)\geq\frac{1}{K_1}e^{-t}\left\lvert x^{-1}y\right\rvert-K_1\geq2K_3^2. 
\end{equation*}
Hence 
\begin{equation*}
d_{h(t)}\left(\partial f(x),\partial f(y)\right)\leq K_3d_t(x,y)+K_3\leq2K_3d_t(x,y)
\end{equation*}
and
\begin{equation*}
d_{h(t)}\left(\partial f(x),\partial f(y)\right)\geq\frac{1}{K_3}d_t(x,y)-\frac{1}{2K_3}2K_3^2\geq\frac{1}{2K_3}d_t(x,y). 
\end{equation*}
\end{proof}

It is easy to show that $h$ is a quasiisometry of $\bb{R}$. See Farb--Mosher \cite[Lemma 5.1]{FM}. 

\begin{lem}
There exists $L>0$ such that for any $t$, $t_0\in\bb{R}$ with $t+L\leq t_0$ we have $h(t)\leq h(t_0)$.
\end{lem}

\begin{proof}
Recall that $h$ is close to $s(x,\cdotp)$ as we saw in \eqref{chikai}. By the definition of $s(x,t)$ we see $s(x,t)\to\pm\infty$ as $t\to\pm\infty$. So $h(t)\to\pm\infty$ as $t\to\pm\infty$. Let $K\geq1$ be a constant such that $\frac{1}{K}\left\lvert s-t\right\rvert-K\leq\left\lvert h(s)-h(t)\right\rvert\leq K\left\lvert s-t\right\rvert+K$ for all $s$, $t\in\bb{R}$. Take $L=4K^2$ and assume the contrary, ie there were $s_0$, $t_0\in\bb{R}$ with $s_0+L\leq t_0$ such that $h(t_0)\leq h(s_0)$. We have $\left\lvert h(s_0)-h(t_0)\right\rvert\geq3K$ and $\left\lvert h(s_0)-h(t)\right\rvert\geq3K$ for any $t\geq t_0$. For $t_0\leq t\leq t_0+1$ we have $\left\lvert h(t)-h(t_0)\right\rvert\leq2K$. Hence we must have $h(t)\leq h(s_0)$ for all $t_0\leq t\leq t_0+1$. Now we have $s_0+L\leq t_0+1$ and $h\left(t_0+1\right)\leq h(s_0)$. Hence this time we get $h(t)\leq h(s_0)$ for all $t_0+1\leq t\leq t_0+2$. By repeating we see that $h(t)\leq h(s_0)$ for all $t\geq t_0$, which is a contradiction. 
\end{proof}

\begin{lem}\label{huhu}
For any fixed $t_0\in\bb{R}$, we have 
\begin{equation*}
\frac{1}{K_2^2}e^{h(t_0)-h(t)}\leq\frac{d_{h(t)}\left(\partial f(x),\partial f(y)\right)}{d_{h(t_0)}\left(\partial f(x),\partial f(y)\right)}\leq K_2^2e^{h(t_0)-h(t)}
\end{equation*}
if $t\leq t_0-L$ and 
\begin{equation*}
\left\lvert x^{-1}y\right\rvert>K_1^2K_3e^{-h(0)}\left(\frac{1}{K_1e^{-h(0)}}\left(\frac{K_1}{K_3}+K_3+K_1\right)+\left(K_1^2+1\right)e^{h(t_0)}\right). 
\end{equation*}
\end{lem}

\begin{proof}
If $t\leq t_0-L$ and 
\begin{equation*}
\left\lvert x^{-1}y\right\rvert>K_1^2K_3e^{-h(0)}\left(\frac{1}{K_1e^{-h(0)}}\left(\frac{K_1}{K_3}+K_3+K_1\right)+\left(K_1^2+1\right)e^{h(t_0)}\right), 
\end{equation*}
then 
\begin{align*}
\left\lvert\partial f(x)^{-1}\partial f(y)\right\rvert&\geq\frac{1}{K_1e^{-h(0)}}\left(d_{h(0)}\left(\partial f(x),\partial f(y)\right)-K_1\right)\\
&\geq\frac{1}{K_1e^{-h(0)}}\left(\frac{1}{K_3}d_0(x,y)-K_3-K_1\right)\\
&\geq\frac{1}{K_1e^{-h(0)}}\left(\frac{1}{K_3}\left(\frac{1}{K_1}\left\lvert x^{-1}y\right\rvert-K_1\right)-K_3-K_1\right)\\
&=\frac{1}{K_1^2K_3e^{-h(0)}}\left\lvert x^{-1}y\right\rvert-\frac{1}{K_1e^{-h(0)}}\left(\frac{K_1}{K_3}+K_3+K_1\right)\\
&>\left(K_1^2+1\right)e^{h(t_0)}
\end{align*}
and $h(t)\leq h(t_0)$. So we get the desired inequality by Corollary \ref{mimi}. 
\end{proof}

\begin{lem}\label{tttt}
There exists $C_3>0$ such that for any $t_0\in\bb{R}$ and $t\leq t_0$, we have 
\begin{equation*}
h(t)\geq t-t_0+h(t_0)-C_3. 
\end{equation*}
\end{lem}

\begin{proof}
Fix $t_0$ and take $x$, $y\in N$ with $\left\lvert x^{-1}y\right\rvert$ large enough so that we can apply Corollary \ref{mimi}, Lemma \ref{lulu} and Lemma \ref{huhu}. Then for any $t\leq t_0-L$, we have 
\begin{align*}
\frac{1}{2K_2^2K_3}e^{h(t_0)-h(t)}d_{t_0}(x,y)&\leq\frac{1}{K_2^2}e^{h(t_0)-h(t)}d_{h(t_0)}\left(\partial f(x),\partial f(y)\right)\\
&\leq d_{h(t)}\left(\partial f(x),\partial f(y)\right)\\
&\leq2K_3d_t(x,y)\\
&\leq2K_2^2K_3e^{t_0-t}d_{t_0}(x,y). 
\end{align*}
Hence 
\begin{equation*}
e^{h(t_0)-h(t)}\leq4K_2^4K_3^2e^{t_0-t}. 
\end{equation*}
Taking $\log$ we get 
\begin{equation*}
h(t_0)-h(t)\leq t_0-t+\log\left(4K_2^4K_3^2\right). 
\end{equation*}
Since $h$ is a quasiisometry, $h(t_0)-h(t)-t_0+t$ is bounded above for $t_0-L\leq t\leq t_0$ by a constant independent of $t_0$. Hence the claim is proved. 
\end{proof}

Let $\bar{f}\colon AN\to AN$ be a coarse inverse of $f$, ie $\bar{f}$ is a quasiisometry such that $\bar{f}\circ f$ and $f\circ\bar{f}$ are close to the identity map. Let $\bar{h}\colon\bb{R}\to\bb{R}$ be a coarse inverse of $h$. It is easy to show that $\bar{f}$ is fiber respecting over $\bar{h}$. Apply Lemma \ref{tttt} to $\bar{f}$ and $\bar{h}$ rather than $f$ and $h$. Then there exists $C_3^\prime>0$ such that 
\begin{equation*}
\bar{h}(s)\geq s-s_0+\bar{h}(s_0)-C_3^\prime
\end{equation*}
for all $s\leq s_0$. Now we can argue completely in the same way as in Farb--Mosher \cite[page 167 just after Claim 5.9]{FM} to prove that $h$ is close to the identity map.

\section{Necessary conditions for parameter rigidity}\label{necessary conditions}
From this section we consider necessary conditions for parameter rigidity. (For the definition of parameter rigidity, see the beginning of Section \ref{introd}.) These necessary conditions are given by certain vanishing of zeroth and first cohomology of the orbit foliation. The main results are Theorem \ref{h^0} and Theorem \ref{h^1}. 

Let $M\stackrel{\rho_0}{\curvearrowleft}S$ denote a $C^\infty$ locally free action of a connected simply connected solvable Lie group $S$ on a closed $C^\infty$ manifold $M$, with the orbit foliation $\mca{F}$. 

Recall that a connected simply connected solvable Lie group $S$ is called {\em of exponential type} if the exponential map $\exp\colon\mf{s}\to S$ is a diffeomorphism, or equivalently, every eigenvalue of $\ad X$ either is $0$ or has nonzero real part for each $X\in\mf{s}$. For a proof of this equivalence, see Dixmier \cite[Th\'{e}or\`{e}me 3]{D} or Saito \cite{Sa}. A derivation of a Lie algebra is called {\em an outer derivation} if it is not an inner derivation. 

The first necessary condition is the following. 

\begin{thm}[Vanishing of $H^0$]\label{h^0}
Assume that $S$ is of exponential type and there is an outer derivation of $\mf{s}$. If $M\stackrel{\rho_0}{\curvearrowleft}S$ is parameter rigid, then $M$ is connected and $H^0(\mca{F})=H^0(\mf{s})$. 
\end{thm}

We will prove Theorem \ref{h^0} in Section \ref{hoo}. 

\begin{cor}
Let $N\neq1$ be a connected simply connected nilpotent Lie group and let $M\stackrel{\rho_0}{\curvearrowleft}N$ be a parameter rigid action. Then $M$ is connected and $H^0(\mca{F})=H^0(\mf{n})$. 
\end{cor}

\begin{proof}
Every nonzero nilpotent Lie algebra over any field has an outer derivation. See Jacobson \cite{J}. 
\end{proof}

Note that $H^0(\mca{F})$ consists of real valued leafwise constant $C^\infty$ functions on $M$ and $H^0(\mf{s})$ (as a subspace of $H^0(\mca{F})$) consists of real valued constant functions on $M$. Hence we have $H^0(\mca{F})=H^0(\mf{s})$ if and only if leafwise constant $C^\infty$ functions are constant. This is satisfied if there is a dense leaf of $\mca{F}$. In the proof of Theorem \ref{h^0} we don't prove the existence of a dense leaf of $\mca{F}$. We prove $H^0(\mca{F})=H^0(\mf{s})$ somewhat algebraically without studying dynamical properties of the foliation $\mca{F}$. 

\begin{rem}
The author does not know whether Theorem \ref{h^0} remains true if we drop one of the two assumptions on $S$. One possibility of constructing counterexamples which are parameter rigid but $H^0(\mca{F})$ is huge is the following. Take a connected simply connected solvable Lie group $S$ and a cocompact lattice $\Gamma$ in $S$ such that: 
\begin{itemize}
\item $S$ has no outer automorphisms
\item $\Gamma$ is a rigid lattice in $S$, which means, if $\Gamma^\prime$ is a lattice in $S$ and $\alpha\colon\Gamma\to\Gamma^\prime$ is an isomorphism, then $\alpha$ extends to an automorphism of $S$. (This terminology is taken from Starkov \cite{S}.)
\end{itemize}
The author does not know the existence of such $S$ and $\Gamma$. But if we had such a pair, Proposition 6.1.2 in Maruhashi \cite{Ma2} says, the action $\Gamma\backslash S\curvearrowleft S$ defined by right multiplication is parameter rigid because in this case parameter rigidity is equivalent to the rigidity of the lattice $\Gamma$. Then the action $S^1\times\Gamma\backslash S\curvearrowleft S$ defined by $(x,y)s=(x,ys)$ is perhaps parameter rigid by the first condition, whereas $H^0(\mca{F})$ is now identified with the space of all real valued $C^\infty$ functions on $S^1$. 
\end{rem}

Recall the following theorem. 

\begin{thm}[Maruhashi \cite{Ma}]\label{nil}
Let $N$ be a connected simply connected nilpotent Lie group, and let $M\stackrel{\rho_0}{\curvearrowleft}N$ be a $C^\infty$ locally free action. Then the following are equivalent: 
\begin{itemize}
\item The action $\rho_0$ is parameter rigid and $H^0(\mca{F})=H^0(\mf{n})$. 
\item $H^1(\mca{F})=H^1(\mf{n})$. 
\end{itemize}
\end{thm}

Hence we have the following. 

\begin{cor}
Let $N$ be a connected simply connected nilpotent Lie group, and let $M\stackrel{\rho_0}{\curvearrowleft}N$ be a $C^\infty$ locally free action. Then the following are equivalent: 
\begin{itemize}
\item The action $\rho_0$ is parameter rigid. 
\item $H^1(\mca{F})=H^1(\mf{n})$. 
\end{itemize}
\end{cor}

\begin{proof}
This is true even if $N=1$. 
\end{proof}

If we have vanishing of $H^0$ for the trivial coefficient, then we can deduce vanishing of $H^0$ for various nontrivial coefficients by an easy argument. This will be done in Lemma \ref{var} in Section \ref{hoo}. 

The second necessary condition is on vanishing of $H^1$. The following will be proved in Section \ref{hone}. 

\begin{thm}[Vanishing of $H^1$]\label{h^1}
Let $V\subset\mf{s}$ be an $\ad$-invariant subspace (ie an ideal of $\mf{s}$) for which $\mf{n}\stackrel{\ad}{\curvearrowright}V$ is trivial. Assume that any eigenvalue of $\ad X$ on $\mf{s}/V$ either is $0$ or has nonzero real part for any $X\in\mf{s}$. If $M\stackrel{\rho_0}{\curvearrowleft}S$ is parameter rigid, then we have 
\begin{equation*}
H^1\left(\mca{F};\mf{s}\stackrel{\ad}{\curvearrowright}V\right)=H^0(\mca{F})\otimes H^1\left(\mf{s};\mf{s}\stackrel{\ad}{\curvearrowright}V\right). 
\end{equation*}
\end{thm}

Note that the assumption is weaker than the assumption that $S$ is of exponential type, as it allows $\ad X\colon V\to V$ to have purely imaginary nonzero eigenvalues. 

Here an element $[\omega]\in H^1\left(\mca{F};\mf{s}\stackrel{\ad}{\curvearrowright}V\right)$ is in $H^0(\mca{F})\otimes H^1\left(\mf{s};\mf{s}\stackrel{\ad}{\curvearrowright}V\right)$ if and only if $[\omega]$ is represented by {\em a leafwise constant form}, that is, represented by a form $\phi\circ\omega_0$ for some $C^\infty$ leafwise constant map $\phi\colon M\to\Hom(\mf{s},V)$. If we assume also that $\mf{s}$ has an outer derivation, then by Theorem \ref{h^0}, the conclusion simplifies to $H^1\left(\mca{F};\mf{s}\stackrel{\ad}{\curvearrowright}V\right)=H^1\left(\mf{s};\mf{s}\stackrel{\ad}{\curvearrowright}V\right)$. 

Let us consider the coefficients appearing Theorem \ref{h^1}. We have $V\subset\mf{n}$, thus $V$ is contained in the center of $\mf{n}$, and is an abelian ideal of $\mf{s}$. (For the first part, if not, take $X\in V\setminus\mf{n}$, then $\mf{n}+\bb{R}X$ would be a nilpotent ideal of $\mf{s}$ which is larger than the nilradical $\mf{n}$.) 

As an example of a coefficient $V$ satisfying the property, we can take $V=\mf{n}^s$, where $\mf{n}\supset\mf{n}^2\supset\cdots\supset\mf{n}^s\supset0$ is the lower central series of $\mf{n}$. 

As a more concrete example, we consider the $2$-dimensional solvable Lie algebra $\mf{ga}=\bb{R}X\oplus\bb{R}Y$ defined by $[X,Y]=Y$. Then the $1$-dimensional representation $\mf{ga}\stackrel{\ad}{\curvearrowright}\bb{R}Y$ satisfies the condition of Theorem \ref{h^1}, but the trivial representation $\mf{s}\curvearrowright\mf{ga}/\bb{R}Y$ does not satisfy the condition.

\section{Vanishing of $H^0$---proof of Theorem \ref{h^0}}\label{hoo}
The proof of Theorem \ref{h^0} is immediate after proving Lemma \ref{out}, whose proof is the main part of this section. Several lemmas before Lemma \ref{out} prepare an ``integration'' map $\mu$, which will be used in the proof of Lemma \ref{out}. Sublemma \ref{sublemma} inside Lemma \ref{out} is similar to Lemma \ref{pp} in the next section and the same kind of argument already appeared in Maruhashi \cite{Ma} when the vanishing of $H^1$ was proved under the assumption of parameter rigidity together with the vanishing of $H^0$ for actions of nilpotent Lie groups. 

Let $M\stackrel{\rho_0}{\curvearrowleft}S$ be a $C^\infty$ locally free action of a connected simply connected solvable Lie group $S$ on a closed $C^\infty$ manifold $M$, with the orbit foliation $\mca{F}$ and the canonical $1$-form $\omega_0$. 

Let $\mf{s}\stackrel{\pi}{\curvearrowright}V$ be a finite dimensional real representation, and let $S\stackrel{\Pi}{\curvearrowright}V$ denote the representation whose differentiation is $\pi$. Then the trivial bundle $M\times V\to M$ is an $S$-equivariant vector bundle with the action defined by 
\begin{equation*}
(x,v)s=\left(\rho_0(x,s),\Pi\left(s^{-1}\right)v\right). 
\end{equation*}
Let $\Gamma_{blc}(V)$ be the space of all bounded sections of $M\times V\to M$ which are continuous on each leaf. (An element $\xi\in\Gamma_{blc}(V)$ can be discontinuous on $M$.) We have a representation $S\curvearrowright\Gamma_{blc}(V)$ by 
\begin{equation*}
\left(s\xi\right)(x)=\Pi(s)\xi\left(\rho_0(x,s)\right)
\end{equation*}
for $s\in S$, $\xi\in\Gamma_{blc}(V)$ and $x\in M$. We equip $V$ with a norm coming from an inner product. Then $\Gamma_{blc}(V)$ is a Banach space with the supremum norm. Let $\Gamma_{lc}(V)$ be the closed subspace of $\Gamma_{blc}(V)$ which consists of bounded leafwise constant sections. 

\begin{lem}\label{contraction}
There is an $S$-equivariant continuous linear map 
\begin{equation*}
\mu\colon\Gamma_{blc}(V)\to\Gamma_{lc}(V)
\end{equation*}
which is the identity on $\Gamma_{lc}(V)$. 
\end{lem}

\begin{proof}
Since $S$ is amenable, by one of the characterizations of amenability, we have a bi-invariant mean $\mu_0\colon C_b(S)\to\bb{R}$ on the space $C_b(S)$ of all bounded continuous real valued functions on $S$. See Page 26--29 of Greenleaf \cite{G}. Recall that $\mu_0(1)=1$ and its operator norm is $1$. Take a basis $v_1,\ldots,v_n$ of $V$. For $\xi=\sum_{i=1}^nf_iv_i\in\Gamma_{blc}(V)$ and $x\in M$, we define 
\begin{equation*}
\mu(\xi)(x)=\sum_{i=1}^n\mu_0\left(f_i\left(\rho_0(x,\cdotp)\right)\right)v_i. 
\end{equation*}
Then this is independent of a choice of a basis of $V$. We have 
\begin{equation*}
\mu(\xi)\left(\rho_0(x,s)\right)=\sum_{i=1}^n\mu_0\left(f_i\left(\rho_0(x,s\ \cdot\ )\right)\right)v_i=\mu(\xi)(x)
\end{equation*}
by left invariance, and
\begin{equation*}
\mu(s\xi)(x)=\Pi(s)\sum_{i=1}^n\mu_0\left(f_i\left(\rho_0(x,\cdot\ s)\right)\right)v_i=s\mu(\xi)(x)
\end{equation*}
by right invariance. We also have $\mu(\xi)=\xi$ for $\xi\in\Gamma_{lc}(V)$ since $\mu_0(1)=1$. By taking $v_1,\ldots,v_n$ to be an orthonormal basis and using $\left\Vert\mu_0\right\Vert=1$, we see 
\begin{equation*}
\left\lVert\mu(\xi)(x)\right\rVert^2\leq\sum_{i=1}^n\left\lVert f_i\right\rVert_\infty^2\leq n\left\lVert\xi\right\rVert_\infty^2. 
\end{equation*}
\end{proof}

Let $\nabla$ denote the flat leafwise connection of $M\times V\to M$ defined by $\mf{s}\stackrel{\pi}{\curvearrowright}V$. 

\begin{lem}
For $v\in V$, $x_0\in M$ and sufficiently small $s\in S$, the locally defined section 
\begin{equation*}
\xi_0\left(\rho_0(x_0,s)\right)=\left(\rho_0(x_0,s),\Pi\left(s^{-1}\right)v\right)
\end{equation*}
of $M\times V\to M$ on the leaf containing $x_0$, is a parallel section for $\nabla$, that is, $\nabla\xi_0=0$. 
\end{lem}

\begin{proof}
For any $y=\rho_0(x_0,s_0)$ with small $s_0\in S$ and any $X\in\mf{s}$, we have 
\begin{align*}
\nabla_{\left.\frac{d}{dt}\rho_0\left(y,e^{tX}\right)\right\vert_{t=0}}\xi_0&=\df\xi_0\left(\left.\frac{d}{dt}\rho_0\left(y,e^{tX}\right)\right\vert_{t=0}\right)+\pi(X)\xi_0(y) \\
&=\left.\frac{d}{dt}\Pi\left(e^{-tX}s_0^{-1}\right)v\right\vert_{t=0}+\pi(X)\Pi\left(s_0^{-1}\right)v \\
&=0. 
\end{align*}
\end{proof}

Therefore, the directions of orbits of the action $M\times V\curvearrowleft S$ are horizontal for the leafwise connection $\nabla$. By the expression of covariant derivative by parallel transport, we have 
\begin{align*}
\left(\nabla_X\xi\right)(x)&=\lim_{t\rightarrow0}\frac{\Pi\left(e^{tX}\right)\xi\left(\rho_0\left(x,e^{tX}\right)\right)-\xi(x)}{t} \\
&=\lim_{t\rightarrow0}\frac{\left(e^{tX}\xi\right)(x)-\xi(x)}{t}
\end{align*}
for any $\xi\in\Gamma(V)$, $X\in\mf{s}$ and $x\in M$. Note that $X\in\mf{s}$ is regarded as $X\in\Gamma(T\mca{F})$ using the locally free action $\rho_0$. 

\begin{lem}\label{uniformly}
For any $\xi\in\Gamma(V)$ and $X\in\mf{s}$, $\frac{e^{tX}\xi-\xi}{t}$ converges uniformly to $\nabla_X\xi$ as $t\rightarrow0$. 
\end{lem}

\begin{proof}
Take a basis $v_1,\ldots,v_n$ of $V$ and write $\left(e^{tX}\xi\right)(x)=\sum_{i=1}^nf_i(t,x)v_i$ for some $C^\infty$ functions $f_i\colon\bb{R}\times M\to\bb{R}$. Then we have $\left(\nabla_X\xi\right)(x)=\sum_{i=1}^nf_i^\prime(0,x)v_i$. The function $f_i(t,x)$ has the Taylor expansion 
\begin{equation*}
f_i(t,x)=f_i(0,x)+tf_i^\prime(0,x)+\frac{t^2}{2}f_i^{\prime\prime}\left(\theta_{i,x,t},x\right), 
\end{equation*}
where $\theta_{i,x,t}$ is a number between $0$ and $t$. Since 
\begin{equation*}
\frac{\left(e^{tX}\xi\right)(x)-\xi(x)}{t}-\left(\nabla_X\xi\right)(x)=\frac{t}{2}\sum_{i=1}^nf_i^{\prime\prime}\left(\theta_{i,x,t},x\right)v_i
\end{equation*}
and $f_i^{\prime\prime}(\theta,x)$ is bounded for $-1\leq\theta\leq 1$ and $x\in M$, we get the conclusion. 
\end{proof}

\begin{lem}\label{commute}
Let $\mu\colon\Gamma_{blc}(V)\to\Gamma_{lc}(V)$ be the map in Lemma \ref{contraction}. Then 
\begin{equation*}
\mu\left(\nabla_X\xi\right)=\nabla_X\mu(\xi)
\end{equation*}
for all $\xi\in\Gamma(V)$ and $X\in\mf{s}$. (Note that $\mu(\xi)$ might be discontinuous on $M$.)
\end{lem}

\begin{proof}
By Lemma \ref{uniformly}, $\frac{e^{tX}\xi-\xi}{t}$ converges uniformly to $\nabla_X\xi$ as $t\rightarrow0$. By continuity and equivariance, we have 
\begin{equation*}
\mu\left(\nabla_X\xi\right)=\lim_{t\rightarrow0}\frac{e^{tX}\mu(\xi)-\mu(\xi)}{t}=\nabla_X\mu(\xi). 
\end{equation*}
\end{proof}

\begin{lem}\label{out}
Assume that $S$ is of exponential type. Let $\Psi\colon M\to\Aut(S)$ be a $C^\infty$ map which is constant on each leaf of $\mca{F}$. If $\rho_0$ is parameter rigid, then $\ol{\Psi}\colon M\to\Out(S)$ is constant on $M$, where bar denotes the projection $\Aut(S)\to\Out(S)$. In particular, if $\Out(S)\neq1$, $M$ must be connected. 
\end{lem}

\begin{proof}
Define $M\stackrel{\rho}{\curvearrowleft}S$ by $\rho(x,s)=\rho_0\left(x,\Psi_x^{-1}(s)\right)$. This defines an action because $\Psi$ is leafwise constant: 
\begin{align*}
\rho\left(x,ss^\prime\right)&=\rho_0\left(\rho_0\left(x,\Psi_x^{-1}(s)\right),\Psi_x^{-1}\left(s^\prime\right)\right) \\
&=\rho_0\left(\rho(x,s),\Psi_{\rho(x,s)}^{-1}\left(s^\prime\right)\right) \\
&=\rho\left(\rho(x,s),s^\prime\right). 
\end{align*}
Since $\rho$ is a $C^\infty$ locally free action with the same orbit foliation as $\rho_0$, $\rho$ is parameter equivalent to $\rho_0$ by parameter rigidity. Note that $\Psi_{x*}\omega_0$ is the canonical $1$-form of $\rho$. By Proposition 1.4.4 of Asaoka \cite{A}, there exist $\Phi\in\Aut(S)$ and a $C^\infty$ map $P\colon M\to S$ such that 
\begin{equation}\label{psi}
\Psi_{x*}\omega_0=\Ad\left(P^{-1}\right)\Phi_*\omega_0+P^*\Theta, 
\end{equation}
where $\Theta$ denotes the left Maurer--Cartan form of $S$. (In \cite{A}, $\Phi$ is referred to as an endomorphism, but it is the same $\Phi$ appearing in the definition of parameter equivalence which we saw in Section \ref{introd}, so $\Phi$ can be taken as an automorphism. It is easy to see $P^*\Theta$ is equivalent to the expression $P^{-1}\df P$ in \cite{A}. There is a small difference between our definition of parameter equivalence and the one in \cite{A}, since in \cite{A} the map $F$ is assumed to be homotopic to the identity through diffeomorphisms. But this does not cause any problem here.)

Let $a$ denote both projections $\mf{s}\to\mf{s}/\mf{n}$ and $S\to S/N$, where $\mf{n}$ is the nilradical of $\mf{s}$ and $N$ is the Lie subgroup corresponding to $\mf{n}$. By projecting \eqref{psi}, we get 
\begin{equation}\label{psi1}
a\Psi_{x*}\omega_0=a\Phi_*\omega_0+\df aP, 
\end{equation}
since $\mf{s}/\mf{n}$ is abelian. For any $x\in M$, $X\in\mf{s}$ and $T>0$, we integrate \eqref{psi1} over the curve $\rho_0\left(x,e^{tX}\right)$ for $0\leq t\leq T$. Then noting $\Psi$ being leafwise constant, 
\begin{equation*}
Ta\Psi_{x*}X=Ta\Phi_*X+aP\left(\rho_0\left(x,e^{TX}\right)\right)-aP(x). 
\end{equation*}
Since $aP$ is bounded due to the compactness of $M$, we must have $a\Psi_{x*}X=a\Phi_*X$ and $aP$ is leafwise constant. Hence there exists a leafwise constant $C^\infty$ map $R\colon M\to S$ such that $Q=R^{-1}P\colon M\to N$. Since $R$ is leafwise constant, we have 
\begin{equation*}
P^*\Theta=(RQ)^*\Theta=Q^*\Theta
\end{equation*}
and \eqref{psi} becomes 
\begin{equation}\label{psi2}
\Psi_{x*}\omega_0=\Ad\left(Q^{-1}\right)\Ad\left(R^{-1}\right)\Phi_*\omega_0+Q^*\Theta. 
\end{equation}

Let $\mf{n}\supset\mf{n}^2\supset\cdots\supset\mf{n}^s\supset0$ be the lower central series of $\mf{n}$. Recall that $\exp\colon\mf{n}\to N$ is a diffeomorphism and $\log\colon N\to\mf{n}$ is defined. 

\begin{sublem}\label{sublemma}
Assume that there exist a $C^\infty$ map $Q\colon M\to N$ and a leafwise constant $C^\infty$ map $R\colon M\to S$ such that: 
\begin{itemize}
\item $\Psi_{x*}\omega_0=\Ad\left(Q^{-1}\right)\Ad\left(R^{-1}\right)\Phi_*\omega_0+Q^*\Theta$
\item $\log Q\in\mf{n}^{k}$ for some $1\leq k\leq s$. 
\end{itemize}
Then we can find a $C^\infty$ map $Q^\prime\colon M\to N$ and a leafwise constant $C^\infty$ map $R^\prime\colon M\to S$ such that: 
\begin{itemize}
\item $\Psi_{x*}\omega_0=\Ad\left(\left(Q^\prime\right)^{-1}\right)\Ad\left(\left(R^\prime\right)^{-1}\right)\Phi_*\omega_0+\left(Q^\prime\right)^*\Theta$
\item $\log Q^\prime\in\mf{n}^{k+1}$. 
\end{itemize}
\end{sublem}

\begin{proof}
Take subspaces $V_0,\ldots,V_s$ such that $\mf{s}=V_0\oplus\mf{n}$ and $\mf{n}^i=V_i\oplus\mf{n}^{i+1}$ for $i=1,\ldots,s$. We can write $Q=\exp\left(\sum_{i=k}^sQ_i\right)$ for some $C^\infty$ maps $Q_i\colon M\to V_i$. We will calculate the $V_k$ component of 
\begin{equation}\label{kkkk}
\Psi_{x*}\omega_0=\Ad\left(Q^{-1}\right)\Ad\left(R^{-1}\right)\Phi_*\omega_0+Q^*\Theta. 
\end{equation}
First note that 
\begin{equation*}
Q^*\Theta\equiv\df Q_k\ \ \ \ \text{mod $\mf{n}^{k+1}$}. 
\end{equation*}
In fact, for all $X=\left.\frac{d}{dt}x(t)\right\vert_{t=0}\in T_x\mca{F}$, 
\begin{align*}
Q^*\Theta(X)&=\left.\frac{d}{dt}Q(x)^{-1}Q\left(x(t)\right)\right\vert_{t=0} \\
&=\left.\frac{d}{dt}\exp\left(-\sum_{i=k}^sQ_i(x)\right)\exp\left(\sum_{i=k}^sQ_i\left(x(t)\right)\right)\right\vert_{t=0} \\
&=\left.\frac{d}{dt}\exp\left(\sum_{i=k}^s\left(Q_i\left(x(t)\right)-Q_i(x)\right)+\text{an element of $\mf{n}^{k+1}$}\right)\right\vert_{t=0} \\
&=\left.\frac{d}{dt}\exp\left(Q_k\left(x(t)\right)-Q_k(x)+\text{an element of $\mf{n}^{k+1}$}\right)\right\vert_{t=0} \\
&\equiv\df Q_k(X)\ \ \ \ \text{mod $\mf{n}^{k+1}$}. 
\end{align*}

Let $\mf{s}\stackrel{\pi_k^0}{\curvearrowright}V_k$ be the representation obtained from $\mf{s}\stackrel{\ad}{\curvearrowright}\mf{n}^k/\mf{n}^{k+1}$ by the identification $V_k\simeq\mf{n}^k/\mf{n}^{k+1}$. Put $\pi_k=\pi_k^0\circ\Phi_*$. We take $\mf{s}\stackrel{\pi_k}{\curvearrowright}V_k$ as $\mf{s}\stackrel{\pi}{\curvearrowright}V$ considered in the beginning of this section; we let $\nabla$ be the leafwise connection defined by $\pi_k$, and we let $\mu\colon\Gamma_{blc}(V_k)\to\Gamma_{lc}(V_k)$ be the map in Lemma \ref{contraction}. 

Write $\Psi_{x*}\omega_0=\sum_{i=0}^s\alpha_i$ and $\Ad\left(R^{-1}\right)\Phi_*\omega_0=\sum_{i=0}^s\beta_i$ according to the decomposition $\mf{s}=\bigoplus_{i=0}^sV_i$. Then we have 
\begin{align*}
\Ad\left(Q^{-1}\right)\Ad\left(R^{-1}\right)\Phi_*\omega_0&=\exp\left(\ad\left(-\sum_{i=k}^sQ_i\right)\right)\sum_{i=0}^s\beta_i \\
&\equiv\sum_{i=0}^k\beta_i+\left[\beta_0,Q_k\right]\ \ \ \ \text{mod $\mf{n}^{k+1}$} \\
&\equiv\sum_{i=0}^k\beta_i+\pi_k^0\beta_0Q_k\ \ \ \ \text{mod $\mf{n}^{k+1}$}. 
\end{align*}

Take the $V_k$ components of \eqref{kkkk} to get 
\begin{equation*}
\alpha_k=\beta_k+\pi_k^0\beta_0Q_k+\df Q_k. 
\end{equation*}
Since 
\begin{equation*}
\Phi_*\omega_0\equiv\Ad\left(R^{-1}\right)\Phi_*\omega_0\equiv\beta_0 \ \ \ \ \text{mod $\mf{n}$}
\end{equation*}
and $\pi_k^0$ vanishes on $\mf{n}$, we have $\pi_k\omega_0=\pi_k^0\beta_0$. Therefore, 
\begin{align*}
\nabla Q_k&=\df Q_k+\pi_k\omega_0Q_k \\
&=\df Q_k+\pi_k^0\beta_0Q_k, 
\end{align*}
and we get 
\begin{equation}\label{same}
\alpha_k=\beta_k+\nabla Q_k. 
\end{equation}
Hence 
\begin{equation*}
\alpha_k(X)=\beta_k(X)+\nabla_XQ_k
\end{equation*}
for any $X\in\mf{s}$. Note that $\alpha_k(X)$ and $\beta_k(X)$ are leafwise constant because $R$ is leafwise constant. Applying $\mu$ and using Lemma \ref{commute}, we get 
\begin{equation*}
\alpha_k(X)=\beta_k(X)+\nabla_X\mu(Q_k). 
\end{equation*}
Therefore, 
\begin{equation*}
\nabla_X\left(Q_k-\mu(Q_k)\right)=0. 
\end{equation*}
Put $Q_k^\prime=Q_k-\mu(Q_k)$. We shall see $Q_k^\prime$ is leafwise constant. Let $S\stackrel{\Pi_k}{\curvearrowright}V_k$ be the representation with the derivative $\mf{s}\stackrel{\pi_k}{\curvearrowright}V_k$. Then for any $t\in\bb{R}$ and $x\in M$, we have 
\begin{align*}
\left.\frac{d}{ds}\right\vert_{s=t}\left(e^{sX}Q_k^\prime\right)(x)&=\lim_{h\to0}\Pi_k\left(e^{tX}\right)\frac{\left(e^{hX}Q_k^\prime\right)\left(\rho_0\left(x,e^{tX}\right)\right)-Q_k^\prime\left(\rho_0\left(x,e^{tX}\right)\right)}{h} \\
&=\Pi_k\left(e^{tX}\right)\left(\nabla_XQ_k^\prime\right)\left(\rho_0\left(x,e^{tX}\right)\right) \\
&=0. 
\end{align*}
Thus $\left(e^{tX}Q_k^\prime\right)(x)=\Pi_k\left(e^{tX}\right)Q_k^\prime\left(\rho_0\left(x,e^{tX}\right)\right)$ is constant with respect to $t$. So 
\begin{equation*}
Q_k^\prime\left(\rho_0\left(x,e^{tX}\right)\right)=\Pi_k\left(e^{-tX}\right)Q_k^\prime(x)=e^{-t\pi_k(X)}Q_k^\prime(x)
\end{equation*}
for all $t\in\bb{R}$. Note that $Q_k^\prime\left(\rho_0\left(x,e^{tX}\right)\right)$ is bounded with respect to $t$. Take a basis of $V_k$ which turns $-\pi_k(X)=-\pi_k^0\left(\Phi_*X\right)$ into a real Jordan normal form. Since any eigenvalue of $\ad X\colon\mf{s}\to\mf{s}$ for any $X\in\mf{s}$ either is $0$ or has nonzero real part by our assumption that $\mf{s}$ is of exponential type, the same is true for $\pi_k^0(X)\colon V_k\to V_k$ for all $X\in\mf{s}$. Therefore, each Jordan block of $-\pi_k(X)=-\pi_k^0(\Phi_*X)$ has the eigenvalue which either is $0$ or has nonzero real part. For a Jordan block whose eigenvalue has the nonzero real part, the corresponding components of $e^{-t\pi_k(X)}Q_k^\prime(x)$ have the following forms: 
\begin{equation*}
\begin{pmatrix}
e^{ta}&&*\\
&\ddots&\\
0&&e^{ta}
\end{pmatrix}
\begin{pmatrix}
c_1\\
\vdots\\
c_m
\end{pmatrix}
\end{equation*}
if the eigenvalue $a$ is real, and 
\begin{equation*}
\begin{pmatrix}
e^{ta}R_t&&*\\
&\ddots&\\
0&&e^{ta}R_t
\end{pmatrix}
\begin{pmatrix}
c_1\\
\vdots\\
c_m
\end{pmatrix}
, 
\end{equation*}
where
\begin{equation*}
R_t=
\begin{pmatrix}
\cos tb&\sin tb\\
-\sin tb&\cos tb
\end{pmatrix}
\end{equation*}
if the eigenvalue $a+bi$ is not real. Since this must be bounded for all $t\in\bb{R}$, $c_1=\cdots=c_m=0$, which implies the corresponding components of $Q_k^\prime\left(\rho_0\left(x,e^{tX}\right)\right)$ must be constant. 

On the other hand, for a Jordan block with the eigenvalue $0$, the corresponding components in $e^{-t\pi_k(X)}Q_k^\prime(x)$ is 
\begin{equation*}
\begin{pmatrix}
1&&*\\
&\ddots&\\
0&&1
\end{pmatrix}
\begin{pmatrix}
c_1\\
\vdots\\
c_m
\end{pmatrix}
, 
\end{equation*}
where the entries in the $*$ part of the matrix are now polynomials in $t$. Since bounded polynomial functions must be constant, we see the corresponding components in $Q_k^\prime\left(\rho_0\left(x,e^{tX}\right)\right)$ are also constant. 

So $Q_k$ is leafwise constant. Put $Q^\prime=e^{-Q_k}Q$. Then $\log Q^\prime$ has values in $\mf{n}^{k+1}$ and 
\begin{equation*}
\Psi_{x*}\omega_0=\Ad\left(\left(Q^\prime\right)^{-1}\right)\Ad\left(\left(R^\prime\right)^{-1}\right)\Phi_*\omega_0+\left(Q^\prime\right)^*\Theta, 
\end{equation*}
where $R^\prime=Re^{Q_k}$ is leafwise constant. 
\end{proof}

Applying Sublemma \ref{sublemma} to \eqref{psi2} repeatedly, we finally get $Q=1$ and therefore 
\begin{equation*}
\Psi_{x*}\omega_0=\Ad\left(R^{-1}\right)\Phi_*\omega_0
\end{equation*}
for some $R$. Therefore, $\Psi_x$ is equal to $\Phi$ modulo inner automorphisms. 
\end{proof}

Theorem \ref{h^0} is restated and proved here. 

\begin{thm}
Assume that $S$ is of exponential type and there is an outer derivation of $\mf{s}$. If $M\stackrel{\rho_0}{\curvearrowleft}S$ is parameter rigid, then $M$ is connected and $H^0(\mca{F})=H^0(\mf{s})$. 
\end{thm}

\begin{proof}
Since there is an outer derivation of $\mf{s}$, the outer automorphism group $\Out(S)$ of $S$ is nontrivial, hence $M$ is connected. Take an outer derivation $\varphi$ of $\mf{s}$ and set $\Phi_t=e^{t\varphi}\in\Aut(S)$. For any $f\in H^0(\mca{F})$, consider a map $M\to\Aut(S)$ defined by $x\mapsto\Phi_{f(x)}$. Since this is leafwise constant, $x\mapsto\overline{\Phi_{f(x)}}\in\Out(S)$ is constant by Lemma \ref{out}. Let $\Inn(S)$ denote the inner automorphism group of $S$. This is a connected normal Lie subgroup of $\Aut(S)$. We must be a bit careful because $\Inn(S)$ might not be closed in $\Aut(S)$ in general. See Hochschild \cite{Hoch}. But the cosets of $\Inn(S)$ defines a foliation on $\Aut(S)$ and $\Phi_t$ is a curve transverse to the foliation. Since the automorphisms $\Phi_{f(x)}$ for all $x\in M$ are contained in a single leaf of $\mca{F}$ and $M$ is connected, $\Phi_{f(x)}$ must be constant with respect to $x$. This implies $f$ is constant over $M$. 
\end{proof}

Finally we see vanishing of $H^0$ with nontrivial coefficients. 

\begin{lem}\label{var}
Assume $H^0(\mca{F})=H^0(\mf{s})$. Let $\mf{s}\stackrel{\pi}{\curvearrowright}V$ be a representation for which $\pi(X)$ has no nonzero purely imaginary eigenvalues for each $X\in\mf{s}$. Then $H^0(\mca{F};\pi)=H^0(\mf{s};\pi)$. 
\end{lem}

\begin{proof}
Take $\xi\in H^0(\mca{F};\pi)$. The function $\xi$ satisfies $\df\xi+\pi\omega_0\xi=0$. This means $X\xi+\pi(X)\xi=0$ for all $X\in\mf{s}$. For each $x\in M$ this is solved as $\xi\left(\rho_0\left(x,e^{tX}\right)\right)=e^{-t\pi(X)}\xi(x)$ for all $t\in\bb{R}$. As in the proof of Sublemma \ref{sublemma} we transform $\pi(X)$ into a real Jordan normal form and $\xi$ being bounded implies $\xi\left(\rho_0\left(x,e^{tX}\right)\right)$ must be constant. Therefore, $\xi$ is leafwise constant. By the assumption $H^0(\mca{F})=H^0(\mf{s})$, $\xi$ is constant on $M$. Hence $\xi\in V$ and $\pi(X)\xi=0$ for all $X\in\mf{s}$, which shows $\xi\in H^0(\mf{s};\pi)$. 
\end{proof}

\section{Vanishing of $H^1$---proof of Theorem \ref{h^1}}\label{hone}
Here we prove the following (a restatement of Theorem \ref{h^1}). 

\begin{thm}
Let $V\subset\mf{s}$ be an $\ad$-invariant subspace (ie an ideal of $\mf{s}$) for which $\mf{n}\stackrel{\ad}{\curvearrowright}V$ is trivial. Assume that any eigenvalue of $\ad X$ on $\mf{s}/V$ either is $0$ or has nonzero real part for any $X\in\mf{s}$. If $M\stackrel{\rho_0}{\curvearrowleft}S$ is parameter rigid, then we have 
\begin{equation*}
H^1\left(\mca{F};\mf{s}\stackrel{\ad}{\curvearrowright}V\right)=H^0(\mca{F})\otimes H^1\left(\mf{s};\mf{s}\stackrel{\ad}{\curvearrowright}V\right). 
\end{equation*}
\end{thm}

\begin{proof}
Take any $[\omega]\in H^1\left(\mca{F};\mf{s}\stackrel{\ad}{\curvearrowright}V\right)$. Let $\omega_0$ be the canonical $1$-form of $\rho_0$. Fix an $\epsilon>0$ and put $\eta:=\omega_0+\epsilon\omega\in\Gamma\left(\Hom\left(T\mca{F},\mf{s}\right)\right)$. Let us see $\eta$ satisfies the Maurer--Cartan equation. As we saw in Section \ref{necessary conditions}, $V$ is abelian and then 
\begin{align*}
\df\eta+[\eta,\eta]
&=\df\omega_0+\epsilon\df\omega+[\omega_0,\omega_0]+\epsilon\left([\omega_0,\omega]+[\omega,\omega_0]\right)\\
&=\epsilon\left(\df\omega+[\omega_0,\omega]+[\omega,\omega_0]\right). 
\end{align*}
But this is zero because $\omega$ satisfies $\df\omega+(\ad\omega_0)\wedge\omega=0$ and $(\ad\omega_0)\wedge\omega=[\omega_0,\omega]+[\omega,\omega_0]$. 

Since $M$ is compact, we can assume $\eta_x\colon T_x\mca{F}\to\mf{s}$ is bijective for all $x\in M$ by taking $\epsilon>0$ small enough. Then there exists a unique action $\rho$ of $S$ on $M$ whose orbit foliation is $\mca{F}$ and whose canonical $1$-form is $\eta$. See Asaoka \cite[Proposition 1.4.3]{A}. By parameter rigidity, $\rho$ is parameter equivalent to $\rho_0$. Thus by Proposition 1.4.4 of \cite{A}, there exist a $C^{\infty}$ map $P\colon M\to S$ and an automorphism $\Phi$ of $S$ satisfying 
\begin{equation}\label{ant}
\omega_0+\epsilon\omega=\Ad\left(P^{-1}\right)\Phi_*\omega_0+P^*\Theta, 
\end{equation}
where $\Theta$ is the left Maurer--Cartan form of $S$. By seeing this equation modulo $\mf{n}$, we get 
\begin{equation*}
\omega_0\equiv\Phi_*\omega_0+\df\ol{P} \ \ \ \ \text{mod $\mf{n}$}, 
\end{equation*}
where bar denotes the projection $S\to S/N$. The same argument as in the proof of vanishing of $H^0$ yields 
\begin{equation*}
\omega_0\equiv\Phi_*\omega_0 \ \ \ \ \text{mod $\mf{n}$}, 
\end{equation*}
\begin{equation*}
\df\ol{P}\equiv0 \ \ \ \ \text{mod $\mf{n}$}. 
\end{equation*}
So we can take a leafwise constant $C^\infty$ map $R\colon M\to S$ such that $Q:=R^{-1}P\in N$. Then Equation \eqref{ant} becomes 
\begin{align*}
\omega_0+\epsilon\omega&=\Ad\left(Q^{-1}R^{-1}\right)\Phi_*\omega_0+\left(RQ\right)^*\Theta \nonumber \\ 
&=\Ad\left(Q^{-1}\right)\Psi_*\omega_0+Q^*\Theta, 
\end{align*}
where $\Psi_*=\Ad\left(R^{-1}\right)\Phi_*$ is leafwise constant. 

\begin{lem}
There exists a filtration
\begin{equation*}
\mf{s}\supset\mf{n}=W_1\supset W_2\supset\cdots\supset W_s=V\supset W_{s+1}=0, 
\end{equation*}
where $W_i$'s are ideals of $\mf{s}$ such that $[\mf{n},W_i]\subset W_{i+1}$. 
\end{lem}

\begin{proof}
If 
\begin{equation*}
\mf{n}\supset\mf{n}^2\supset\cdots\supset\mf{n}^{s-1}\supset0
\end{equation*}
denotes the lower central series of $\mf{n}$, then the filtration 
\begin{equation*}
\mf{s}\supset\mf{n}\supset\mf{n}^2+V\supset\mf{n}^3+V\supset\cdots\supset\mf{n}^{s-1}+V\supset V\supset0
\end{equation*}
gives the desired filtration. 
\end{proof}

Note that we have $\omega_0\equiv\Psi_*\omega_0$ modulo $W_1$. 

\begin{lem}\label{pp}
Assume there exist a $C^\infty$ map $Q\colon M\to N$ and a leafwise constant $C^\infty$ map $\Psi\colon M\to\Aut(S)$ such that 
\begin{equation}\label{apple}
\omega_0+\epsilon\omega=\Ad\left(Q^{-1}\right)\Psi_*\omega_0+Q^*\Theta, 
\end{equation}
\begin{equation*}
\log Q\in W_k
\end{equation*}
and 
\begin{equation*}
\omega_0\equiv\Psi_*\omega_0 \ \ \ \ \text{mod $W_k$: }
\end{equation*}
\begin{enumerate}
\item If $k<s$, then there exist a $C^\infty$ map $Q^\prime\colon M\to N$ and a leafwise constant $C^\infty$ map $\Psi^\prime\colon M\to\Aut(S)$ such that 
\begin{equation*}
\omega_0+\epsilon\omega=\Ad\left(\left(Q^\prime\right)^{-1}\right)\Psi^\prime_*\omega_0+\left(Q^\prime\right)^*\Theta, 
\end{equation*}
\begin{equation*}
\log Q^\prime\in W_{k+1}
\end{equation*}
and 
\begin{equation*}
\omega_0\equiv\Psi^\prime_*\omega_0 \ \ \ \ \text{mod $W_{k+1}$}. 
\end{equation*}
\item If $k=s$, then $\omega$ is cohomologous to a leafwise constant cocycle. 
\end{enumerate}
\end{lem}

\begin{proof}
The proof is similar to the proof of Sublemma \ref{sublemma}. 
Take complementary subspaces $V_i$'s so that $\mf{s}=V_0\oplus\mf{n}$ and $W_i=V_i\oplus W_{i+1}$. Write 
\begin{gather*}
\omega_0=\sum_{i=0}^s\alpha_i,\quad\Psi_*\omega_0=\sum_{i=0}^s\beta_i\quad\text{and}\quad Q=\exp\left(\sum_{i=k}^sQ_i\right)
\end{gather*}
according to the decomposition $\mf{s}=\bigoplus_{i=0}^sV_i$. 

The same calculation as in Sublemma \ref{sublemma} gives 
\begin{equation*}
Q^*\Theta\equiv\df Q_k \ \ \ \ \text{mod $W_{k+1}$}. 
\end{equation*}
We have 
\begin{align*}
\Ad\left(Q^{-1}\right)\Psi_*\omega_0&=\exp\left(\ad\left(-\sum_{i=k}^sQ_i\right)\right)\sum_{i=0}^s\beta_i \\
&\equiv\sum_{i=0}^k\beta_i+[\beta_0,Q_k]\ \ \ \ \text{mod $W_{k+1}$} \\ 
&=\sum_{i=0}^{k-1}\alpha_i+\beta_k+[\alpha_0,Q_k]\ \ \ \ \text{mod $W_{k+1}$}.
\end{align*}
Equation \eqref{apple} gives 
\begin{equation*}
\sum_{i=0}^k\alpha_i+\delta_{ks}\epsilon\omega\equiv\sum_{i=0}^{k-1}\alpha_i+\beta_k+[\alpha_0,Q_k]+\df Q_k \ \ \ \ \text{mod $W_{k+1}$}. 
\end{equation*}
Thus 
\begin{equation*}
\alpha_k+\delta_{ks}\epsilon\omega\equiv\beta_k+[\alpha_0,Q_k]+\df Q_k \ \ \ \ \text{mod $W_{k+1}$}. 
\end{equation*}
If $k=s$, we have  
\begin{equation*}
\omega=\epsilon^{-1}\left(\beta_s-\alpha_s\right)+\df\left(\epsilon^{-1}Q_s\right)+\left[\alpha_0,\epsilon^{-1}Q_s\right]. 
\end{equation*}
If $\nabla$ denotes the covariant derivative defined from $\mf{s}\stackrel{\ad}{\curvearrowright}V$, then by $[\mf{n},V]=0$ we have 
\begin{align*}
\nabla\left(\epsilon^{-1}Q_s\right)&=\df\left(\epsilon^{-1}Q_s\right)+\left[\omega_0,\epsilon^{-1}Q_s\right]\\
&=\df\left(\epsilon^{-1}Q_s\right)+\left[\alpha_0,\epsilon^{-1}Q_s\right]. 
\end{align*}
Therefore, $\omega$ is cohomologous to $\epsilon^{-1}\left(\beta_s-\alpha_s\right)$ which is leafwise constant since so are $\omega_0$ and $\Psi_*\omega_0$. 

If $k<s$, then 
\begin{equation*}
\alpha_k\equiv\beta_k+[\alpha_0,Q_k]+\df Q_k \ \ \ \ \text{mod $W_{k+1}$}. 
\end{equation*}
Let $\mf{s}\stackrel{\pi_k}{\curvearrowright}V_k$ denote the representation obtained from $\mf{s}\stackrel{\ad}{\curvearrowright}W_k/W_{k+1}$ by the identification $W_k/W_{k+1}\simeq V_k$, and let $\nabla$ be the leafwise connection defined by $\pi_k$. Recall that $\nabla Q_k=\df Q_k+\pi_k\omega_0Q_k$. Since 
\begin{align*}
\pi_k\omega_0Q_k&=\pi_k\left(\sum_{i=0}^s\alpha_i\right)Q_k \\
&\equiv[\alpha_0,Q_k] \ \ \ \ \text{mod $W_{k+1}$}, 
\end{align*}
we have 
\begin{equation*}
\alpha_k\equiv\beta_k+\df Q_k+\pi_k\omega_0Q_k \ \ \ \ \text{mod $W_{k+1}$}, 
\end{equation*}
which implies 
\begin{align*}
\alpha_k&=\beta_k+\df Q_k+\pi_k\omega_0Q_k \\
&=\beta_k+\nabla Q_k. 
\end{align*}
By the same argument starting from Equation \eqref{same} in the proof of vanishing of $H^0$, using the assumption on the eigenvalues of $\ad X$, we can conclude that $Q_k$ is leafwise constant. Define $Q^\prime\colon M\to N$ by $Q=e^{Q_k}Q^\prime$. Then Equation \eqref{apple} becomes 
\begin{align*}
\omega_0+\epsilon\omega&=\Ad\left(\left(Q^\prime\right)^{-1}e^{-Q_k}\right)\Psi_*\omega_0+\left(e^{Q_k}Q^\prime\right)^*\Theta \\ 
&=\Ad\left(\left(Q^\prime\right)^{-1}\right)\Psi^\prime_*\omega_0+\left(Q^\prime\right)^*\Theta, 
\end{align*}
where $\Psi^\prime_*=\Ad\left(e^{-Q_k}\right)\Psi_*$. Now we have $\log Q^\prime\in W_{k+1}$ and 
\begin{align*}
\Psi^\prime_*\omega_0&=e^{-\ad Q_k}\Psi_*\omega_0 \\
&=e^{-\ad Q_k}\left(\sum_{i=0}^{k-1}\alpha_i+\beta_k+\text{an element of $W_{k+1}$}\right) \\
&\equiv\sum_{i=0}^{k-1}\alpha_i+\beta_k+[\alpha_0,Q_k] \ \ \ \ \text{mod $W_{k+1}$} \\
&\equiv\sum_{i=0}^k\alpha_i \ \ \ \ \text{mod $W_{k+1}$} \\
&=\omega_0 \ \ \ \ \text{mod $W_{k+1}$}
\end{align*}
since $\df Q_k=0$. 
\end{proof}

Applying Lemma \ref{pp} repeatedly, we see that $\omega$ is cohomologous to a leafwise constant cocycle. Note that we have used the assumption on the eigenvalues only on $V_1,\ldots,V_{s-1}$, but not on $V_s=V$. 
\end{proof}

\section*{Acknowledgements}
This paper was written during the stays at Institut des Hautes \'{E}tudes Scientifiques in Bures-sur-Yvette and Max-Planck-Institut f\"{u}r Mathematik in Bonn. The stays were supported by Researcher Exchange Program between Japan Society for the Promotion of Science and those institutes.

\bibliography{maruhashi}
\end{document}